\documentclass{amsart}
\usepackage{graphicx}
\usepackage{algorithm}
\usepackage{algorithmic}

\newtheorem{theorem}{Theorem}[section]
\newtheorem{lemma}[theorem]{Lemma}

\theoremstyle{definition}
\newtheorem{definition}[theorem]{Definition}
\newtheorem{example}[theorem]{Example}

\newtheorem{corollary}[theorem]{Corollary}
\theoremstyle{remark}
\newtheorem{remark}[theorem]{Remark}

\numberwithin{equation}{section}

\begin{document}

\title[Acceleration of Clenshaw-Curtis quadrature]{Convergence rate and acceleration of Clenshaw-Curtis quadrature for functions with endpoint singularities}

\author{Haiyong Wang}
\address{School of Mathematics and Statistics, Huazhong University of Science and
Technology, Wuhan 430074, P. R. China}
\email{haiyongwang@hust.edu.cn}
\thanks{The author was supported by the National Science Foundation of China (No.~11301200).}






\begin{abstract}
In this paper, we study the rate of convergence of Clenshaw-Curtis
quadrature for functions with endpoint singularities in $X^s$, where
$X^s$ denotes the space of functions whose Chebyshev coefficients
decay asymptotically as $a_k = \mathcal{O}(k^{-s-1})$ for some
positive $s$. For such a subclass of $X^s$, we show that the
convergence rate of $(n+1)$-point Clenshaw-Curtis quadrature is
$\mathcal{O}(n^{-s-2})$. Furthermore, an asymptotic error expansion
for Clenshaw-Curtis quadrature is presented which enables us to
employ some extrapolation techniques to accelerate its convergence.
Numerical examples are provided to confirm our analysis.
\end{abstract}

\maketitle

%

\section{Introduction}
The evaluation of the definite integral
\begin{equation}
I[f] := \int_{-1}^{1} f(x) dx,
\end{equation}
is one of the fundamental and important research topics in the field
of numerical analysis \cite{davis1984quadrature}. Given a set of
distinct nodes $\{ x_j \}_{j=0}^{n}$, an interpolatory quadrature
rule of the form
\begin{equation}\label{eq:quadrature}
Q_n[f] := \sum_{j=0}^{n} w_j f(x_j),
\end{equation}
can be constructed to approximate the above integral by requiring
$I[f] = Q_n[f]$ whenever $f(x)$ is a polynomial of degree $n$ or
less. In order to obtain a stable quadrature rule, the quadrature
nodes with the Chebyshev density $\mu(x) = 1/\sqrt{1-x^2}$ are
preferable. Ideal candidates are the roots or extrema of classical
orthogonal polynomials such as Chebyshev and Legendre polynomials.

Clenshaw-Curtis quadrature rule, which is the interpolatory
quadrature formula based on the extrema of Chebyshev polynomials,
has attracted considerable attention in the past few decades. Let
$\{ x_j \}_{j=0}^{n}$ be the Clenshaw-Curtis points or the
Chebyshev-Lobatto points
\begin{equation}\label{eq:chebyshev points}
x_j = \cos\left( \frac{j \pi}{n} \right), \quad j=0,\ldots,n.
\end{equation}
Then the Clenshaw-Curtis quadrature rule is
\begin{equation}
I_{n}^{C}[f] :=  \sum_{j=0}^{n} w_j f(x_j),
\end{equation}
where the quadrature weights are given explicitly by
\cite[p.~86]{davis1984quadrature}
\begin{equation}
w_j = \frac{4\delta_j}{n} \sum_{k=0}^{ [\frac{n}{2}] }
 \frac{\delta_{2k} }{1-4k^2} \cos\left( \frac{2jk\pi}{n} \right),
\end{equation}
and the coefficients $\delta_j$ are defined as
\begin{equation}
\delta_j = \left\{\begin{array}{cc}
                                        1/2,   & \mbox{$\textstyle  j=0$ or
                                          $j=n$},\\ [5pt]
                                          1,   & \mbox{otherwise}.
                                        \end{array}
                                        \right.
\end{equation}
Here $[\cdot]$ denotes the integer part. It is well known that the
quadrature weights are all positive and can be computed in only
$\mathcal{O}(n \log n)$ operations by the inverse Fourier transform
\cite{waldvogel2006clenshaw}.

Clenshaw-Curtis quadrature rule with $(n+1)$-point is exact for
polynomials of degree less than or equal to $n$. However, its
performance for differentiable functions is comparable with the
classic Gauss-Legendre quadrature which is exact for polynomials of
degree up to $2n+1$. This remarkable accuracy makes it
extraordinarily attractive and many studies have been done on the
error behaviour of
the Clenshaw-Curtis quadrature 
(see, for example,
\cite{mason2003chebyshev,OHara1968ccquad,Riess1972ccquadrature,trefethen2008gausscc,trefethen2013atap,xiang2012clenshawcurtis}).
In particular, Trefethen in \cite{trefethen2008gausscc} presented a
comprehensive comparison of error bounds of Gauss and
Clenshaw-Curtis quadrature rules for analytic and differentiable
functions. For the latter, an $\mathcal{O}(n^{-s})$ bound was
established for functions belong to $ X^s$, where $X^s$ denotes the
space of functions whose Chebyshev coefficients decay asymptotically
as $a_k = \mathcal{O}(k^{-s-1})$ for some positive $s$. More
recently, Xiang and Bornemann in \cite{xiang2012clenshawcurtis}
presented a more accurate estimate and showed that the optimal rate
of convergence of Clenshaw-Curtis quadrature rule for $f \in X^s$ is
$\mathcal{O}(n^{-s-1})$.

In this work, we are interested in the rate of convergence of
Clenshaw-Curtis quadrature for the integrals $\int_{-1}^{1} f(x)
dx$, where the integrands $f(x)$ have singularities at one or both
endpoints. More specifically, we assume that
\begin{equation}\label{eq:algebraic functions}
f(x) = (1 - x)^{\alpha} (1 + x)^{\beta} g(x),
\end{equation}
where $\alpha, \beta \geq 0$ are not integers simultaneously and
$g(x) \in C^{\infty}[-1,1]$. Note that the assumption $\alpha, \beta
\geq 0$ is due to the fact that Clenshaw-Curtis quadrature needs to
evaluate the values of the integrand $f(x)$ at both endpoints. When
such kind of functions belong to the space $X^s$ where $s$ is
determined by the strength of singularities of $f$, however, we will
show that the optimal rate of convergence of Clenshaw-Curtis
quadrature for evaluating the integrals $\int_{-1}^{1} f(x) dx$ is
$\mathcal{O}(n^{-s-2})$, which is one power of $n$ better than that
given in \cite{xiang2012clenshawcurtis}. Furthermore, we also extend
our analysis to functions with algebraic-logarithmic endpoint
singularities of the form
\begin{equation}
f(x) = (1 - x)^{\alpha} (1 + x)^{\beta} \log(1 - x) g(x),
\end{equation}
where $\alpha$ is a positive integer and $\beta \geq 0$ and $g(x)
\in C^{\infty}[-1,1]$. Similarly, we show that the optimal rate of
convergence of Clenshaw-Curtis quadrature is also
$\mathcal{O}(n^{-s-2})$ if $f(x)$ belongs to $X^s$.

Apart from the close connection with the FFT, another particularly
significant advantage of Clenshaw-Curtis quadrature is that its
quadrature nodes are nested. This means that it is possible to
accelerate the convergence of Clenshaw-Curtis quadrature by using
some extrapolation schemes. In Section \ref{sec:acceleration}, we
shall explore the asymptotic expansion of the error of
Clenshaw-Curtis quadrature for functions with endpoint
singularities. An asymptotic series in negative powers of $n$ is
derived for even $n$, which allows to employ some extrapolation
schemes, such as the Richardson extrapolation approach, to
accelerate the convergence of Clenshaw-Curtis quadrature. Thus,
comparing with Gauss-Legendre quadrature, Clenshaw-Curtis quadrature
is a more attractive scheme for computing the integrals whose
integrands have endpoint singularities.

The rate of convergence of Gauss-Legendre quadrature for functions
with endpoint singularities has been investigated considerably in
the past decades (see
\cite{chawla1968gauss,lubinsky1984gauss,rabinowitz1968gauss,rabinowitz1986gauss,sidi2009variable,sidi2009gauss,verlinden1997gauss}
and references therein). For example, for functions like $f(x) =
(1-x)^{\alpha} g(x)$ where $\alpha>-1$ is not an integer and $g(x)$
is sufficiently smooth, Rabinowitz in
\cite{rabinowitz1968gauss,rabinowitz1986gauss} and Luninsky and
Rabinowitz in \cite{lubinsky1984gauss} have shown that the
asymptotic error estimate of the $n$-point Gauss-Legendre quadrature
is $\mathcal{O}(n^{-2\alpha-2})$ as $n \rightarrow \infty$. On the
other hand, Verlinden in \cite{verlinden1997gauss} and Sidi in
\cite{sidi2009gauss} further studied the asymptotic expansion of the
error of the Gauss-Legendre quadrature for functions with algebraic
and algebraic-logarithmic endpoint singularities. Although the rate
of convergence and asymptotic error expansion of Gauss-Legendre
quadrature for functions with endpoint singularities have been
extensively explored, we are still unable to find the corresponding
result for the Clenshaw-Curtis quadrature in the literature. This
motivates the author to conduct the current research.

The rest of the paper is organized as follows. In the next section,
we shall show that the rate of convergence of Clenshaw-Curtis
quadrature can be improved to $\mathcal{O}(n^{-s-2})$ if the
Chebyshev coefficients of functions in $X^s$ satisfy a more specific
condition; see Theorem \ref{thm:superconvergence of cc} for details.
In Section \ref{sec:asymptotic of Chebyshev coeff} we discuss the
asymptotic behaviour of Chebyshev coefficients of functions with
endpoints singularities, including algebraic and
algebraic-logarithmic singularities. An asymptotic error expansion
for Clenshaw-Curtis quadrature is presented in Section
\ref{sec:acceleration}. This allows us to use some extrapolation
schemes for convergence acceleration. We present some numerical
examples in Section \ref{sec:example} and give some concluding
remarks in Section \ref{sec:conclusion}.

\section{Conditions for enhanced convergence rate}\label{sec:improved rate of clenshaw}
In this section, we establish sufficient conditions under which the
rate of convergence of Clenshaw-Curtis quadrature for functions in
$X^s$ can be further enhanced. We commence our analysis from a
helpful lemma.
\begin{lemma}\label{lem:estimate H}
For each $k \geq 1$, we have
\begin{equation}\label{eq:estimate H}
\sum_{r=1}^{n} \frac{r^{2k}}{4r^2 - 1} = \frac{1}{4^{k-1}}
\frac{n(n+1)}{2(2n+1)} + \sum_{j=1}^{2k-1} \nu_j^{k} n^{2k-j} ,
\end{equation}
where
\begin{equation}\label{eq:last coeff H}
\nu_{2j+1}^{k} =  \frac{1}{\Gamma(2k-2j)} \sum_{p=1}^{j+1}
\frac{\Gamma(2k-2p+1) }{ \Gamma(2j-2p+3) } \frac{B_{2j-2p+2} }{
4^{p}} , \quad 0 \leq j \leq k-2,
\end{equation}
and
\begin{equation}\label{eq:last coeff H}
\nu_{2k-1}^{k} = \sum_{p=1}^{k-1} \frac{1}{4^{p} } B_{2k-2p}.
\end{equation}
Here $B_j$ denotes the $j$-th Bernoulli number ($B_0 = 1, B_2 =
\frac{1}{6}, \ldots$). Moreover,
\begin{equation}\label{eq:expansion coeff H}
\nu_{2j}^{k} = \frac{1}{2^{2j+1}} , \quad 1 \leq j \leq k-1.
\end{equation}
\end{lemma}
\begin{proof}
Let $H(n,k)$ denote the sum on the left hand side of
\eqref{eq:estimate H}. It is easy to derive the following recurrence
relation
\begin{equation}\label{eq:rec H}
4 H(n,j+1) = H(n,j) + \sum_{r=1}^{n} r^{2j}.
\end{equation}
Let $S(n,j)$ denote the last sum on the right hand side of the above
equation. Multiplying both sides of the above equation by $4^{j-1}$
and summing over $j$ from $1$ to $k-1$, we obtain
\begin{equation}\label{eq:formula H}
H(n,k) = \frac{1}{4^{k-1}} H(n,1) + \sum_{j=1}^{k-1}
\frac{1}{4^{k-j}} S(n,j),
\end{equation}
where the sum on the right hand side vanishes when $k=1$. For
$H(n,1)$, straightforward computation gives
\begin{equation}\label{eq:H one}
H(n,1) = \frac{n(n+1)}{2(2n+1)}.
\end{equation}
Moreover, using the Faulhaber's formula
\cite[Corollary~3.4]{javed2013trapezoidal}, we have
\begin{equation}\label{eq:faulhaber}
S(n,k) = \frac{n^{2k+1}}{2k+1} + \frac{n^{2k}}{2} + \sum_{j=1}^{k}
\frac{\Gamma(2k+1) B_{2j}}{\Gamma(2j+1) \Gamma(2k-2j+2)}
n^{2k-2j+1}.
\end{equation}
Substituting \eqref{eq:H one} and \eqref{eq:faulhaber} into
\eqref{eq:formula H} gives the desired result.
\end{proof}

In the following we shall present sufficient conditions for the
enhanced rate of convergence of Clenshaw-Curtis quadrature.

\begin{theorem}\label{thm:superconvergence of cc}
Suppose $f \in X^s$ and if the Chebyshev coefficients of $f(x)$
decay asymptotically as
\begin{equation}\label{eq:cheb coefficients constant sign}
a_{m} = \frac{c(s)}{ m^{s+1}} + \mathcal{O}(m^{-s-2}), \quad m\geq
m_0,
\end{equation}
or
\begin{equation}\label{eq:cheb coefficients alternate sign}
a_m = (-1)^m \frac{c(s)}{m^{s+1}} + \mathcal{O}(m^{-s-2}), \quad
m\geq m_0,
\end{equation}
where $c(s)$ is independent of $m$. Then, for $n\geq \max\{m_0,
2\}$, the rate of convergence of Clenshaw-Curtis quadrature rule can
be improved to
\begin{equation}\label{eq:rate of cc}
E_n^{C}(f) = \mathcal{O}(n^{-s-2}).
\end{equation}
\end{theorem}
\begin{proof} In \cite{xiang2012clenshawcurtis}, the authors have
presented a simple and elegant proof on the rate of convergence of
Clenshaw-Curtis quadrature. For the sake of clarity, we shall
briefly describe their idea and then give the key observation that
leads to \eqref{eq:rate of cc}.

Define
\begin{equation}
\Delta(n) = \{ m~ |~ m = 2jn + 2r,~ j\geq 1, ~ 1-n \leq 2r \leq n
\}.
\end{equation}
Note that the Clenshaw-Curtis rule is exact for polynomials of
degree $n$ and $E_n^{C}(f) = 0$ for odd functions $f$. The error of
the Clenshaw-Curtis quadrature rule can be written as
\begin{equation}\label{eq:error of cc}
E_n^{C}(f) = \sum_{m \in \Delta(n) } a_m E_n^{C}(T_m),
\end{equation}
where $T_j(x)$ denotes the Chebyshev polynomial of degree $j$.
Moreover, using the aliasing condition, we have that
\begin{equation}\label{eq:aliasing}
E_n^{C}(T_{m}) = \frac{2}{ 1 - m^2 } - \frac{2}{1-4r^2}, \quad m\in
\Delta(n).
\end{equation}
Substituting this into the reminder $E_n^{C}(f)$ yields
\[
E_n^{C}(f) = S_1 + S_2,
\]
where
\begin{equation}
S_1 = \sum_{m \in \Delta(n)} \frac{2a_{m} }{ 1 - m^2 }, \quad  S_2 =
\sum_{m \in \Delta(n)} \frac{2a_{m} }{ 4r^2 - 1 }.
\end{equation}
From the assumption that $f\in X^s$, it is easy to deduce that $S_1
= \mathcal{O}(n^{-s-2})$. The remaining task is to give an accurate
estimate of $S_2$. Using the following identities
\begin{equation}
\sum_{r=-\infty}^{\infty} \frac{1}{|4r^2-1|} = 2, \quad
\sum_{j=1}^{\infty} \frac{1}{j^{s+1}} = \zeta(s+1),
\end{equation}
where $\zeta(n)$ is the Riemann zeta function, Xiang and Bornemann
in \cite{xiang2012clenshawcurtis} deduced that
\begin{eqnarray}\label{eq:estimate s2}
| S_2 |  & \leq & \sum_{m \in \Delta(n)} \frac{2 |a_{m} | }{ | 4r^2 - 1 | } \nonumber \\
& = & \sum_{j=1}^{\infty}  \sum_{ 1-n \leq 2r \leq n} \frac{2 |
a_{2j n + 2r} | }{ |4r^2 - 1| } = \mathcal{O}(n^{-s-1}).
\end{eqnarray}
Hence, they proved that the convergence rate of the Clenshaw-Curtis
quadrature for $f \in X^s$ is $\mathcal{O}(n^{-s-1})$.

In the following, we shall show that if $f\in X^s$ and (\ref{eq:cheb
coefficients constant sign}) or (\ref{eq:cheb coefficients alternate
sign}) is satisfied, the rate of convergence of the Clenshaw-Curtis
quadrature can be further improved. The key observation is that the
estimate of $S_2$ can be further improved to
$\mathcal{O}(n^{-s-2})$. Here we only discuss the case (\ref{eq:cheb
coefficients constant sign}) and the case (\ref{eq:cheb coefficients
alternate sign}) can be analyzed similarly.

For $n \geq \max\{m_0, 2\}$, substituting the asymptotic of $a_m$
into $S_2$, we have
\begin{eqnarray}
S_2 & = & \sum_{ m\in \Delta(n) }
\frac{2a_{m} }{ 4r^2 - 1 } \nonumber \\
& = &  \sum_{ m\in \Delta(n) } \frac{2 c(s) }{ (4r^2 - 1) m^{s+1} }
+ \sum_{ m\in \Delta(n) }
\frac{2 }{ (4r^2 - 1)} \mathcal{O}(m^{-s-2}).   
\end{eqnarray}
In analogy to the estimate of \eqref{eq:estimate s2}, it is easy to
deduce that the last sum in the above equation is
$\mathcal{O}(n^{-s-2})$, and thus we get
\begin{align}\label{eq:expansion of s2}
S_2 & = \sum_{ m\in \Delta(n) } \frac{2 c(s) }{ (4r^2 - 1)
m^{s+1} } +  \mathcal{O}(n^{-s-2}) \nonumber \\
& =  \sum_{j=1}^{\infty}  \sum_{ 1-n \leq 2r \leq n} \frac{ 2 c(s)
}{(4r^2 - 1)(2jn+2r)^{s+1}} +
\mathcal{O}(n^{-s-2}) \nonumber \\
& =  \frac{2c(s)}{(2n)^{s+1}} \sum_{j=1}^{\infty} \frac{1}{j^{s+1}}
\sum_{ 1-n \leq 2r \leq n} \frac{1}{ 4r^2 - 1 } \left(1 + \frac{r}{
j n } \right)^{-s-1}  +  \mathcal{O}(n^{-s-2}).
\end{align}
We now consider the asymptotic behaviour of the double sum in the
above equation. First, we consider the case that $n$ is even.
Rearranging the inner sum, we obtain
\begin{align}\label{eq:rearrange sum even}
& \sum_{ 1-n \leq 2r \leq n} \frac{1}{ 4r^2 - 1} \left(1 +
\frac{r}{ j n } \right)^{-s-1} \nonumber \\
& = -1 + \sum_{k=1}^{n/2} \frac{1}{4k^2 - 1} \left[ \left( 1 +
\frac{k}{jn} \right)^{-s-1}
 +  \left( 1 - \frac{k}{jn} \right)^{-s-1} \right]  \\
 &~~~~~~~~~~~ - \frac{1}{n^2 - 1} \left( 1 - \frac{1}{2j}
 \right)^{-s-1}. \nonumber
\end{align}
Utilizing the following binomial series expansion
\begin{equation}\label{eq:binomial expansion}
(1+x)^{-\beta} = \sum_{k=0}^{\infty} (-1)^k \frac{ (\beta)_k }{k!}
x^k, \quad |x|<1,
\end{equation}
where $(z)_n$ is the Pochhammer symbol, we further get
\begin{align}
\sum_{ 1-n \leq 2r \leq n} \frac{1}{ 4r^2 - 1 } \left(1 + \frac{r}{
j n } \right)^{-s-1} & = -1 + \sum_{k=1}^{n/2}
\frac{2}{4k^2 - 1}  \sum_{q=0}^{\infty} \frac{(s+1)_{2q}}{(2q)!} \left( \frac{k}{jn} \right)^{2q}  \nonumber \\
 &~~~~~~~~~~~ - \frac{1}{n^2 - 1} \left( 1 - \frac{1}{2j}
\right)^{-s-1}.
\end{align}
This together with the following identities
\begin{equation}
\sum_{k=1}^{n/2} \frac{2}{4k^2 - 1} = \frac{n}{n+1}, \quad
\sum_{j=1}^{\infty}  \left( j-\frac{1}{2} \right)^{-s-1} = (2^{s+1}
- 1) \zeta(s+1),
\end{equation}
gives
\begin{align}
S_2 & =\frac{2c(s)}{(2n)^{s+1}} \bigg( - \left(
\frac{2^{s+1}-1}{n^2-1} + \frac{1}{n+1} \right)\zeta(s+1) \nonumber \\
&~~~~~~ + \sum_{j=1}^{\infty} \frac{1}{j^{s+1}} \sum_{k=1}^{n/2}
\frac{2}{4k^2 - 1} \sum_{q=1}^{\infty} \frac{(s+1)_{2q}}{(2q)!}
\left( \frac{k}{jn} \right)^{2q} \bigg)  + \mathcal{O}(n^{-s-2}).
\end{align}
Next, we explore the asymptotic behaviour of the last term inside
the bracket. By the results of Lemma \ref{eq:estimate H}, we have
\begin{align}
&~~~~~ \sum_{j=1}^{\infty} \frac{1}{j^{s+1}} \sum_{k=1}^{n/2}
\frac{2}{4k^2 - 1} \sum_{q=1}^{\infty} \frac{(s+1)_{2q}}{(2q)!}
\left( \frac{k}{jn} \right)^{2q} \nonumber \\
& = 2 \sum_{j=1}^{\infty} \frac{1}{j^{s+1}} \sum_{q=1}^{\infty}
\frac{(s+1)_{2q}}{(2q)! (jn)^{2q} } \left( \frac{n(n+2)}{4^q 2(n+1)
} +  \sum_{k=1}^{2q-1} \nu_{k}^{q} \left(
\frac{n}{2} \right)^{2q-k} \right) \nonumber \\
& = \frac{n(n+2)}{n+1} \sum_{j=1}^{\infty}
\frac{1}{j^{s+1}}\sum_{q=1}^{\infty}
\frac{(s+1)_{2q}}{(2q)!(2jn)^{2q}} \nonumber \\
&~~~ + 2\sum_{j=1}^{\infty} \frac{1}{j^{s+1}} \sum_{q=1}^{\infty}
\frac{(s+1)_{2q}}{(2q)! (jn)^{2q}} \left( \sum_{k=1}^{q-1}
\nu_{2k}^{q} \left( \frac{n}{2} \right)^{2q-2k} + \sum_{k=1}^{q}
\nu_{2k-1}^{q} \left( \frac{n}{2} \right)^{2q-2k+1} \right) .
\nonumber
\end{align}
Now using the explicit expression of $\nu_{k}^{q}$ and after some
elementary computations, we arrive at
\begin{align}
&~~ \sum_{j=1}^{\infty} \frac{1}{j^{s+1}} \sum_{k=1}^{n/2}
\frac{2}{4k^2 - 1} \sum_{q=1}^{\infty} \frac{(s+1)_{2q}}{(2q)!}
\left( \frac{k}{jn} \right)^{2q} \nonumber \\
& = \frac{1}{n^2-1} \sum_{j=1}^{\infty} \frac{1}{j^{s+1}}
\sum_{q=1}^{\infty} \frac{ (s+1)_{2q} }{ (2q)! (2j)^{2q} } \nonumber
\\
&~~~ + \left( \frac{n(n+2)}{n+1} - \frac{n^2}{n^2 - 1} \right)
\sum_{k = 1}^{\infty} \frac{(s+1)_{2k} \zeta(s+2k+1)}{(2k)!
(2n)^{2k}}  \\
& ~~~ + \sum_{k=0}^{\infty} \frac{1}{n^{2k+1}} \sum_{j=1}^{\infty}
\frac{1}{j^{s+1}} \sum_{\ell=0}^{\infty} \frac{
\nu_{2k+1}^{k+\ell+1} (s+1)_{2\ell+2k+2} }{2^{2\ell} (2\ell+2k+2)!
j^{2\ell+2k+2} } \nonumber \\
& = \mathcal{O}(n^{-1}). \nonumber
\end{align}
Hence, we immediately deduce that
\begin{align}
S_2 & = \mathcal{O}(n^{-s-2}), \quad  n \rightarrow\infty.
\end{align}
Thus, the desired result follows. For the case that $n$ is odd,
similar to \eqref{eq:rearrange sum even}, rearranging the summation
yields
\begin{align}
& \sum_{ 1-n \leq 2r \leq n} \frac{1}{ 4r^2 - 1 } \left(1 +
\frac{r}{ j n } \right)^{-s-1} \nonumber \\
& = -1 + \sum_{k=1}^{\frac{n-1}{2}} \frac{1}{4k^2 - 1} \left[ \left(
1 + \frac{k}{jn} \right)^{-s-1}
 +  \left( 1 - \frac{k}{jn} \right)^{-s-1} \right].
\end{align}
The remaining argument can be proceeded similarly as the case $n$ is
even and we omit the details. This proves the theorem.
\end{proof}

\begin{remark}
Functions satisfy \eqref{eq:cheb coefficients constant sign} or
\eqref{eq:cheb coefficients alternate sign} are only a subclass of
$X^s$. We will show in the next section that typical examples are
functions with endpoint singularities.
\end{remark}

\begin{remark}\label{eq:superconv of cc additional term}
If additional terms like
\begin{equation}
b_{m,s} = \pm\frac{d(s)}{m^{s+1+\mu}},~~~ \mbox{ or }~~~ \pm (-1)^m
\frac{d(s)}{m^{s+1+\mu}},
\end{equation}
where $d(s)$ is independent of $m$ and $0 < \mu < 1$, are added in
\eqref{eq:cheb coefficients constant sign} or \eqref{eq:cheb
coefficients alternate sign}. Then, similar to the estimate of
$S_2$, we can deduce that
\begin{equation}
\sum_{ m\in \Delta(n) } \frac{2b_{m,s} }{ 4r^2 - 1 } =
\mathcal{O}(n^{-s-2-\mu}).
\end{equation}
Hence, the rate of convergence of Clenshaw-Curtis quadrature rule is
also $E_n^{C}(f) = \mathcal{O}(n^{-s-2})$.
\end{remark}

\section{Asymptotics of Chebyshev coefficients of functions with endpoints
singularities}\label{sec:asymptotic of Chebyshev coeff} We have
showed that the rate of convergence of Clenshaw-Curtis quadrature
can be improved if the Chebyshev coefficients of $f(x)$ satisfy
\eqref{eq:cheb coefficients constant sign} or \eqref{eq:cheb
coefficients alternate sign}. It is natural to raise the following
question: what kind of functions satisfy these conditions? In this
section we shall give some typical examples, including functions
with algebraic and algebraic-logarithmic singularities. Moreover,
for each class of functions, we also establish the corresponding
rate of convergence of Clenshaw-Curtis quadrature.

\subsection{Functions with algebraic singularities}
Elliott in \cite{elliott1965chebyshev} and Tuan and Elliott in
\cite{tuan1972spectral} have investigated the asymptotic of
Chebyshev coefficients of the following singular functions
\begin{equation}
f(x) = ( 1 \pm x)^{\alpha} g(x),
\end{equation}
where $\alpha>0$ is not an integer and $g(x)$ is analytic in a
region containing the interval $[-1,1]$. For example, for $f(x) = (
1 - x)^{\alpha} g(x)$, it was shown that its Chebyshev coefficients
satisfy \cite[Eqn.~(4.13)]{tuan1972spectral}
\begin{equation}
a_n = - \frac{2^{1-\alpha} g(1) \sin(\alpha\pi) }{\pi n^{2\alpha+1}}
\Gamma(2\alpha+1) + \mathcal{O}(n^{-2\alpha-3}).
\end{equation}
Obviously, functions of this kind satisfy the conditions of Theorem
\ref{thm:superconvergence of cc}. For functions with endpoint
singularities of the following general form
\begin{equation}\label{eq:f both endpoint singularities}
f(x) = (1 - x)^{\alpha} (1 + x)^{\beta} g(x),
\end{equation}
where $\alpha , \beta$ are positive real numbers not integers and
$g(x)$ is analytic in a region containing both endpoints, Tuan and
Elliott in \cite{tuan1972spectral} proposed a complicated technique
to separate the singularities with the aid of auxiliary functions
and then derived the asymptotic of the Chebyshev coefficients. For
more details, we
refer the reader to \cite{tuan1972spectral}. 

In the following we shall present a simpler approach to analyze the
asymptotic of Chebyshev coefficients of $f(x) = (1 - x)^{\alpha} (1
+ x)^{\beta} g(x)$ with $\alpha, \beta
> - \frac{1}{2}$ are not integers simultaneously. Meanwhile, for the
sake of simplicity, we always assume that $g(x) \in
C^{\infty}[-1,1]$. However, the generalization to the case $g(x) \in
C^m[-1,1]$ for some positive integer $m$ is mathematically
straightforward. Note that the assumptions we consider here is more
general than that considered in
\cite{elliott1965chebyshev,tuan1972spectral}.

Before commencing our analysis, we give a useful lemma.
\begin{lemma}\label{lem:asymp for singular oscill integrals}
Suppose that
\begin{align*}
f(x)=(x-a)^{\gamma} (b-x)^{\delta} h(x)
\end{align*}
with $\gamma,\delta > -1$ and $h(x)$ is $m$ times continuously
differentiable for $x \in [a,b]$. Furthermore, define
\begin{align*}
\phi(x) = (x-a)^{\gamma} h(x), \quad  \psi(x)=(b-x)^{\delta} h(x).
\end{align*}
Then for large $\lambda$,
\begin{align}
\int_{a}^{b} f(x) e^{i\lambda x} dx & \sim e^{i \lambda a}
\sum_{k=0}^{m-1}\frac{\psi^{(k)}(a) e^{i \frac{\pi}{2} (k+\gamma+1)}
\Gamma(k+\gamma+1)}{\lambda^{k+\gamma+1}s!}
\nonumber \\
&~~~~~~  - e^{i \lambda b} \sum_{k=0}^{m-1}\frac{\phi^{(k)}(b) e^{i
\frac{\pi}{2} (k-\delta+1)}
\Gamma(k+\delta+1)}{\lambda^{k+\delta+1}s!}   \nonumber \\
&~~~  + \mathcal{O}(\lambda^{-m-1-\min\{\gamma,\delta\}}), \quad
\lambda \rightarrow \infty.
\end{align}
\end{lemma}
\begin{proof}
The first proof of this result was given by Erd\'{e}lyi in
\cite{erdelyi1955fourier}. The idea was based on the neutralizer
functions together with integration by parts
\cite[Thm.~3]{erdelyi1955fourier}. If $h(x)$ is analytic in a
neighborhood of the interval $[a,b]$, an alternative proof based on
the contour integration was given by Lyness
\cite[Thm.~1.12]{lyness1972fourier}.
\end{proof}

We now give the asymptotic of Chebyshev coefficients of functions
with algebraic endpoint singularities. 

\begin{theorem}\label{thm:asymptotic albebraic singularities}
For the function $f(x) = (1 - x)^{\alpha} (1 + x)^{\beta} g(x)$ with
$\alpha, \beta > -\frac{1}{2}$ are not integers simultaneously and
$g(x) \in C^{\infty}[-1,1]$, then its Chebyshev coefficients satisfy
\begin{align}\label{eq:asymptotic expansion Cheby coeff}
a_n & \sim  \frac{2^{\alpha+\beta+1}}{\pi} \bigg\{- \sin(\alpha\pi)
\sum_{k=0}^{\infty} \frac{ (-1)^k \hat{\psi}^{(2k)}(0)
\Gamma(2k+2\alpha+1) }{ n^{2k+2\alpha+1} (2k)! }  \nonumber\\
& - (-1)^n \sin(\beta\pi) \sum_{k=0}^{\infty} \frac{ (-1)^k
\hat{\phi}^{(2k)}(\pi) \Gamma(2k+2\beta+1) }{ n^{2k+2\beta+1} (2k)!
} \bigg\}, \quad n \rightarrow \infty,
\end{align}
where
\begin{equation}\label{eq:auxiliary functions}
\hat{\psi}(t) = (\pi - t )^{2\beta} \hat{g}( t) , \quad
\hat{\phi}(t) =  t^{2\alpha} \hat{g}(t),
\end{equation}
and
\begin{equation}\label{eq:hat g}
\hat{g}(t) = \left( t^{-1} \sin(t/2)\right)^{2\alpha} \left( (\pi -
t)^{-1} \cos(t/2) \right)^{2\beta} g(\cos(t)).
\end{equation}
These values $\hat{\psi}^{(2s)}(0)$ and $\hat{\phi}^{(2s)}(\pi)$ can
be calculated explicitly using the L'H\^{o}pital's rule. Here we
give the first several values
\begin{equation}
\hat{\psi}(0) = \frac{g(1)}{2^{2\alpha}}, \quad \hat{\phi}(\pi) =
\frac{g(-1)}{2^{2\beta}},
\end{equation}
and
\begin{equation}
\hat{\psi}{''}(0) = - \frac{g(1)}{2^{2\alpha+1}} \left(
\frac{\alpha}{3} + \beta \right) - \frac{ g'(1)}{2^{2\alpha}} ,
\quad \hat{\phi}{''}(\pi) = - \frac{g(-1)}{2^{2\beta+1}} \left(
\alpha + \frac{\beta}{3} \right) + \frac{ g'(-1)}{2^{2\beta}}.
\end{equation}
\end{theorem}
\begin{proof}
First, make a change of variable $x = \cos t$, we have
\begin{align}\label{eq:transformed cheby coefficients}
a_n &= \frac{2}{\pi} \int_{0}^{\pi} f(\cos t) \cos(nt) dt \nonumber
\\
& = \frac{2^{\alpha+\beta+1}}{\pi} \int_{0}^{\pi} t^{2\alpha} (\pi -
t)^{2\beta} \hat{g}(t) \cos(nt) dt,
\end{align}
where $ \hat{g}(t) $ is defined as in \eqref{eq:hat g}. It is easy
to see that $\hat{g}(t) \in C^{\infty}[0, \pi]$. On the other hand,
we observe that $\hat{\psi}(t)$ defined in \eqref{eq:auxiliary
functions} is infinitely differentiable at $t=0$ while
$\hat{\phi}(t)$ is infinitely differentiable at $t=\pi$, and
\[
\hat{\psi}(-t) = \hat{\psi}(t), \quad \hat{\phi}(\pi + t) =
\hat{\phi}(\pi - t).
\]
Hence, it holds that
\begin{equation}
\hat{\psi}^{(2k+1)}(0) = 0, \quad \hat{\phi}^{(2k+1)}(\pi) = 0,
\quad k \geq 0.
\end{equation}
The desired result then follows from applying Lemma \ref{lem:asymp
for singular oscill integrals} with $m = \infty$ to the integral
\eqref{eq:transformed cheby coefficients}.

\end{proof}

\begin{remark}
The assumption $\alpha, \beta > - \frac{1}{2}$ can not be relaxed to
$\alpha, \beta > -1$ since the Chebyshev coefficients $a_n$ of the
function $f(x) =  (1-x)^{\alpha} (1+x)^{\beta} g(x)$ will be
divergent if one of $\alpha$ and $\beta$ is less than or equal to
$-\frac{1}{2}$.
\end{remark}

\begin{corollary}
Under the same assumptions as in Theorem \ref{thm:asymptotic
albebraic singularities}, if neither $\alpha$ nor $\beta$ is a
nonegative integer, then the leading term of Chebyshev coefficients
of $f(x)$ is given by
\begin{equation}\label{eq:asymptotic one}
a_n = \left\{\begin{array}{ccc}
                                           {\displaystyle  (-1)^{n+1} \frac{2^{\alpha-\beta+1} g(-1) \sin(\beta\pi)}{\pi
n^{2\beta+1}}\Gamma(2\beta+1) + \mathcal{O}(n^{- \min\{ 2\alpha+1,
2\beta+3 \}}) },   & \mbox{$\alpha>\beta$}, \\ [10pt]
                                           {\displaystyle  - \frac{2 \sin(\alpha\pi) \Gamma(2\alpha+1)}{\pi
n^{2\alpha+1}} \left( g(1) + (-1)^n g(-1) \right) + \mathcal{O}(n^{-2\alpha-3})}, & \mbox{$\alpha = \beta$}, \\
[10pt]
                                           {\displaystyle  - \frac{2^{\beta-\alpha+1} g(1) \sin(\alpha\pi)}{\pi
n^{2\alpha+1}}\Gamma(2\alpha+1)  + \mathcal{O}(n^{- \min\{
2\alpha+3, 2\beta+1 \}})},   & \mbox{$\alpha < \beta$}.
                                        \end{array}
                                        \right.
\end{equation}
Further, if one of $\alpha$ and $\beta$ is an integer, then
\begin{equation}\label{eq:asymptotic two}
a_n = \left\{\begin{array}{cc}
                                           {\displaystyle  (-1)^{n+1} \frac{2^{\alpha-\beta+1} g(-1) \sin(\beta\pi)}{\pi
n^{2\beta+1}}\Gamma(2\beta+1)  + \mathcal{O}(n^{- 2\beta-3 })},   &
\mbox{if $\alpha$ is an integer}, \\ [10pt]
                                           {\displaystyle  - \frac{2^{\beta-\alpha+1} g(1) \sin(\alpha\pi)}{\pi
n^{2\alpha+1}}\Gamma(2\alpha+1) + \mathcal{O}(n^{-2\alpha-3})},   &
\mbox{if $\beta$ is an integer}.
                                        \end{array}
                                        \right.
\end{equation}
\end{corollary}
\begin{proof}
It follows immediately from Theorem \ref{thm:asymptotic albebraic
singularities} by taking the leading term of \eqref{eq:asymptotic
expansion Cheby coeff}.
\end{proof}

Having derived the leading term of the asymptotic of the Chebyshev
coefficients, we can define the parameter $s$ such that $f$ belong
to the space $X^s$. Note that our aim is to establish the rate of
convergence of Clenshaw-Curtis quadrature for the integral
$\int_{-1}^{1} f(x) dx$. Thus we restrict our attention to the case
$\alpha, \beta \geq 0$.


\begin{definition}\label{def: s one}
For the function $f(x) = (1-x)^{\alpha} (1+x)^{\beta} g(x)$ with
$\alpha, \beta \geq 0$ are not integers simultaneously and $g(x) \in
C^{\infty}[-1,1]$. Define
\begin{equation}\label{eq:class s}
s = \left\{\begin{array}{ccc}
                                           {\displaystyle  2\min\{ \alpha, \beta \} },   & \mbox{if $\alpha,\beta$ are not integers}, \\ [10pt]
                                           {\displaystyle  2\alpha}, & \mbox{if $\beta$ is an integer}, \\ [10pt]
                                           {\displaystyle  2\beta},  & \mbox{if $\alpha$ is an integer}.
                                        \end{array}
                                        \right.
\end{equation}
Then, from equations \eqref{eq:asymptotic one} and
\eqref{eq:asymptotic two} we can deduce immediately that $f \in
X^s$.
\end{definition}

\begin{theorem}\label{thm:rate algebraic endpoint}
Let $f(x)$ satisfy the assumptions as in Definition \ref{def: s
one}. Then, the rate of convergence of $(n+1)$-point Clenshaw-Curtis
quadrature for the integral $\int_{-1}^{1} f(x) dx$ is
\begin{equation}
E_n^{C}[f] = \mathcal{O}(n^{-s-2}),
\end{equation}
where $s$ is defined as in \eqref{eq:class s}.
\end{theorem}

\begin{proof}
If one of $\alpha$ and $\beta$ is an integer, then we observe from
equation \eqref{eq:asymptotic two} that the Chebyshev coefficients
satisfy the condition of Theorem \ref{thm:superconvergence of cc},
therefore the desired result holds. If neither $\alpha$ nor $\beta$
is a nonegative integer, then the desired result holds when $\alpha
= \beta$ due to the second equation of \eqref{eq:asymptotic one}. We
now consider the case $\alpha>\beta$: if $\alpha \geq \beta + 1$,
then the desired result follows by noting the first equation of
\eqref{eq:asymptotic one}. If $\beta < \alpha < \beta + 1$, using
Theorem \ref{thm:asymptotic albebraic singularities} we find that
\begin{align*}
a_n & =  (-1)^{n+1} \frac{2^{\alpha-\beta+1} g(-1)
\sin(\beta\pi)}{\pi n^{2\beta+1}}\Gamma(2\beta+1) \\
&~~~~~~~~~~ - \frac{2^{\beta-\alpha+1} g(1) \sin(\alpha\pi)}{\pi
n^{2\alpha+1}}\Gamma(2\alpha+1)  + \mathcal{O}(n^{-2\beta-3}).
\end{align*}
This together with Remark \ref{eq:superconv of cc additional term}
gives the desired result. Thus, the proof is completed since the
argument in the case $\alpha < \beta$ is similar.
\end{proof}

\begin{corollary}
For functions of the form $f(x) = (1 \pm x)^{\alpha} g(x)$, where
$\alpha > 0$ is not an integer and $g(x) \in C^{\infty}[-1,1]$. From
Theorem \ref{thm:rate algebraic endpoint}, we see that the
convergence rate of Clenshaw-Curtis quadrature is $E_n^{C}[f] =
\mathcal{O}(n^{-2\alpha-2})$, which is the same as that of
Gauss-Legendre quadrature.
\end{corollary}

\begin{remark}
Not all functions with algebraic endpoint singularities can be
expressed in terms of the form $f(x) = (1 - x)^{\alpha}
(1+x)^{\beta} g(x)$. Typical examples are $f(x) = \log( 1 +
\sin\sqrt{1 - x})$ and $f(x) = \arccos(x^{2m})$ where $m$ is a
positive integer. The latter function has square root singularities
at $x = \pm 1$. However, if we formally define $f(x) = \sqrt{1 -
x^2} g(x)$ with $g(x) = \arccos(x^{2m})/\sqrt{1 - x^2}$. It is easy
to verify that $g(x) \in C^{\infty}[-1,1]$. Thus, the result of
Theorem \ref{thm:rate algebraic endpoint} still holds for this
function; see Example \ref{example: arecos} for details.
\end{remark}

\subsection{Functions with algebraic-logarithmic singularities}
In this subsection we consider the asymptotic of the Chebyshev
coefficients for functions of the following form
\begin{equation}\label{eq:log singularity}
f(x) = (1 - x)^{\alpha} (1 + x)^{\beta} \log(1 - x) g(x),
\end{equation}
where $\alpha >0 $, $\beta > -\frac{1}{2}$ and $g(x)\in
C^{\infty}[-1,1]$.

\begin{lemma}\label{lem:asymp for logarithm oscillartory integrals}
Suppose that
\begin{align*}
f(x)=(x-a)^{\gamma}(b-x)^{\delta} \log(x-a) h(x)
\end{align*}
with $\gamma > 0, \delta > -1$ and $h(x)$ is $m$ times continuously
differentiable for $x \in [a,b]$. Define
\begin{align*}
\phi(x) = (x-a)^{\gamma} \log(x-a) h(x), \quad
\psi(x)=(b-x)^{\delta} h(x).
\end{align*}
Then for large $\lambda$,
\begin{align}
\int_{a}^{b} f(x) e^{i\lambda x} dx & =  e^{i \lambda a}
\sum_{k=0}^{m-1}\frac{\psi^{(k)}(a) e^{i \frac{\pi}{2} (k+\gamma+1)}
\Gamma(k+\gamma+1)}{\lambda^{k+\gamma+1}k!} \left(
\tilde{\psi}(k+\gamma+1) - \log\lambda + \frac{\pi}{2} i \right)
\nonumber \\
&~~~  - e^{i \lambda b} \sum_{k=0}^{m-1}\frac{\phi^{(k)}(b) e^{i
\frac{\pi}{2} (k-\delta+1)}
\Gamma(k+\delta+1)}{\lambda^{k+\delta+1}k!} \nonumber\\
&~~~~~ + \mathcal{O}( \lambda^{-m-\gamma-1} \log\lambda ) +
\mathcal{O}( \lambda^{-m-\delta-1}), \nonumber
\end{align}
where $\tilde{\psi}(x)$ is the digamma function.
\end{lemma}
\begin{proof}
The idea of Erd\'{e}lyi's proof can be extended to the current
setting in a straightforward way; see \cite{erdelyi1956fourier} for
details. If $h(x)$ is analytic, the desired result can be derived by
using the technique of contour integration
\cite[Appendix]{lyness1972fourier}.
\end{proof}

Using the above Lemma, we obtain the following.
\begin{theorem}\label{thm:asymptotic logarithmic singularity}
For the function $f(x) = (1 - x)^{\alpha} (1 + x)^{\beta} \log(1 -
x) g(x)$ with $\alpha > 0$, $\beta > -\frac{1}{2}$ and $g(x) \in
C^{\infty}[-1,1]$, its Chebyshev coefficients are given
asymptotically by
\begin{align}\label{eq:left logarithm}
a_n &\sim - \frac{2^{\alpha+\beta+1}}{\pi} \sin(\alpha\pi)\sum_{s =
0}^{\infty} \frac{ (-1)^s \psi_1^{(2s)}(0)
\Gamma(2s+2\alpha+1) }{n^{2s+2\alpha+1} (2s)!} \nonumber \\
&~~~~~ -(-1)^n \frac{2^{\alpha+\beta+1}}{\pi} \sin(\beta\pi) \sum_{s
= 0}^{\infty} \frac{(-1)^s \phi_1^{(2s)}(\pi) \Gamma(2s+2\beta+1)
}{n^{2s+2\beta+1} (2s)!} \nonumber \\
&~~~~~ - \frac{2^{\alpha+\beta+2}}{\pi} \sum_{s = 0}^{\infty} \frac{
(-1)^s \hat{\psi}^{(2s)}(0)
\Gamma(2s+2\alpha+1) }{n^{2s+2\alpha+1} (2s)!} \bigg( \sin(\alpha\pi) (\tilde{\psi}(2s+2\alpha+1) - \log n )  \nonumber \\
&~~~~~ + \frac{\pi}{2}\cos(\alpha\pi) \bigg), \quad
n\rightarrow\infty,
\end{align}
where
\begin{equation}\label{def:psi and phi one}
\psi_1(t) = \hat{\psi}(t) \log\left( 2(t^{-1} \sin(t/2) )^2 \right),
\quad  \phi_1(t) = \hat{\phi}(t) \log\left( 2 (\sin(t/2))^2 \right).
\end{equation}
\end{theorem}
\begin{proof} The idea is similar to the proof of Theorem \ref{thm:asymptotic albebraic singularities}. The change of variable $x = \cos
t$ results in
\begin{align}
a_n &= \frac{2}{\pi} \int_{0}^{\pi} f(\cos t) \cos(nt) dt \nonumber
\\
&= \frac{2}{\pi} \int_{0}^{\pi} (1 - \cos t)^{\alpha} (1 + \cos
t)^{\beta} \log(1 - \cos t) g(\cos t) \cos(nt) dt \nonumber
\\
& = \frac{2^{\alpha+\beta+2}}{\pi} \int_{0}^{\pi} t^{2\alpha} (\pi -
t)^{2\beta} \log(t) \hat{g}(t) \cos(nt) dt \nonumber \\
&~~~~~~~~~~~ + \frac{2^{\alpha+\beta+1}}{\pi} \int_{0}^{\pi}
t^{2\alpha} (\pi - t)^{2\beta} \tilde{g}(t) \cos(nt) dt,
\end{align}
where $\hat{g}(t)$ is defined as in \eqref{eq:hat g} and
\begin{equation}
\tilde{g}(t) = \hat{g}(t) \log\left( 2(t^{-1} \sin(t/2) )^2 \right),
\end{equation}
which is infinitely differentiable on $[0,\pi]$. Note that
$\psi_1(t)$ defined in \eqref{def:psi and phi one} is an even
function, we have
\begin{equation}
\psi_1^{(2k+1)}(0) = 0, \quad k \geq 0.
\end{equation}
This together with Lemmas \ref{lem:asymp for logarithm oscillartory
integrals} and \ref{lem:asymp for singular oscill integrals} gives
\begin{align}\label{eq:asymptotic algebraic logarithm}
a_n &\sim \frac{2^{\alpha+\beta+1}}{\pi} \bigg\{ -
\sin(\alpha\pi)\sum_{k = 0}^{\infty} \frac{ (-1)^k \psi_1^{(2k)}(0)
\Gamma(2k+2\alpha+1) }{n^{2k+2\alpha+1} (2k)!} \nonumber \\
&~~~~~ -(-1)^n \sum_{k = 0}^{\infty} \frac{ \phi_2^{(k)}(\pi)
\Gamma(k+2\beta+1) }{n^{k+2\beta+1} k!} \sin\left( \beta\pi -
\frac{k}{2}\pi \right) \bigg\} \nonumber \\
&~~~~~ + \frac{2^{\alpha+\beta+2}}{\pi} \bigg\{ -  \sum_{k =
0}^{\infty} \frac{ (-1)^k \hat{\psi}^{(2k)}(0)
\Gamma(2k+2\alpha+1) }{n^{2k+2\alpha+1} (2k)!} \bigg( \sin(\alpha\pi) (\tilde{\psi}(2k+2\alpha+1) - \log n )  \nonumber \\
&~~~~~ + \frac{\pi}{2}\cos(\alpha\pi) \bigg) -(-1)^n \sum_{k =
0}^{\infty} \frac{\phi_3^{(k)}(\pi) \Gamma(k+2\beta+1)
}{n^{k+2\beta+1} k!} \sin\left(\beta\pi - \frac{k}{2}\pi \right)
\bigg\},
\end{align}
where
\begin{align}
\phi_2(t) = \hat{\phi}(t) \log\left( 2(t^{-1} \sin(t/2) )^2 \right),
\quad \phi_3(t) = \hat{\phi}(t) \log t.
\end{align}
By \eqref{def:psi and phi one}, we get $\phi_1(t) = \phi_2(t) + 2
\phi_3(t)$. On the other hand, for $k \geq 0$, we have
\begin{align*}
\phi_1^{(2k+1)}(\pi) &= \left( \hat{\phi}(t)
\log\left( 2 (\sin(t/2))^2 \right) \right)^{(2k+1)}_{t = \pi} \\
& = \sum_{j=0}^{2k+1} \binom{2k+1}{j} \hat{\phi}^{(j)}(\pi) \left(
\log\left( 2 (\sin(t/2))^2 \right) \right)^{(2k+1-j)}_{t = \pi} \\
& = 0,
\end{align*}
where we have used the fact that
\[
\hat{\phi}^{(2j+1)}(\pi) = 0, \quad  \left( \log\left( 2
(\sin(t/2))^2 \right) \right)^{(2j+1)}(\pi) = 0, \quad j \geq 0.
\]
This together with the second and the last sums on the right hand
side of \eqref{eq:asymptotic algebraic logarithm} gives the desired
result. This completes the proof.
\end{proof}

\begin{corollary}
Under the same assumptions as in Theorem \ref{thm:asymptotic
logarithmic singularity}. If $\alpha$ is a positive integer and
$\beta$ is a nonegative integer, then
\begin{align}\label{eq:logarithm one}
a_n &\sim  - 2^{\alpha+\beta+1} \cos(\alpha\pi)  \sum_{k =
0}^{\infty} \frac{ (-1)^k \hat{\psi}^{(2k)}(0) \Gamma(2k+2\alpha+1)
}{n^{2k+2\alpha+1} (2k)!}.
\end{align}
If $\alpha$ is a positive integer and $\beta$ is not a nonnegative
integer, then
\begin{equation}\label{eq:logarithm two}
a_n = \left\{\begin{array}{ccc}
                                           {\displaystyle  \frac{ (-1)^{n+1} 2^{\alpha-\beta+1} g(-1)
\sin(\beta\pi)}{\pi n^{2\beta+1}}\Gamma(2\beta+1) \log2 +
\mathcal{O}(n^{- \min\{ 2\alpha+1, 2\beta+3 \}}) }, & \mbox{$\alpha > \beta$}, \\
[10pt]
                                           {\displaystyle  - \frac{2^{\beta-\alpha+1} g(1) \cos(\alpha\pi)}{
n^{2\alpha+1}}\Gamma(2\alpha+1)  + \mathcal{O}(n^{- \min\{
2\alpha+3, 2\beta+1 \}})},   & \mbox{$\alpha < \beta$}.
                                        \end{array}
                                        \right.
\end{equation}
\end{corollary}
\begin{proof}
It follows directly from Theorem \ref{thm:asymptotic logarithmic
singularity}.
\end{proof}

Again, we define the parameter $s$ such that $f \in X^s$. Meanwhile,
we restrict our attention to the case $\alpha > 0$ and $\beta \geq
0$ since our aim is to derive the optimal rate of convergence of
Clenshaw-Curtis quadrature for the integral $\int_{-1}^{1} f(x) dx$.
\begin{definition}\label{def: s two}
For the function $f(x) = (1 - x)^{\alpha} (1 + x)^{\beta} \log(1 -
x) g(x)$ with $\alpha$ a positive integer and $\beta \geq 0$ and
$g(x) \in C^{\infty}[-1,1]$. Define
\begin{equation}\label{def:s logarithm}
s = \left\{\begin{array}{ccc}
                                           {\displaystyle  2\alpha}, & \mbox{if $\beta$ is an integer}, \\ [10pt]
                                           {\displaystyle  2 \min\{\alpha, \beta \} },  & \mbox{otherwise}.
                                        \end{array}
                                        \right.
\end{equation}
From the above corollary we see that $f \in X^s$.
\end{definition}
\begin{theorem}
Let $f(x)$ satisfy the assumptions as in Definition \ref{def: s
two}. Then, the rate of convergence of $(n+1)$-point Clenshaw-Curtis
quadrature for the integral $\int_{-1}^{1} f(x) dx$ is
\begin{equation}
E_n^{C}[f] = \mathcal{O}(n^{-s-2}),
\end{equation}
where $s$ is defined as in \eqref{def:s logarithm}.
\end{theorem}
\begin{proof}
The proof is similar to the proof of Theorem \ref{thm:rate algebraic
endpoint}.
\end{proof}

\begin{remark}
For the case that $\alpha$ is not a positive integer, from
\eqref{eq:left logarithm} we observe that there exists a factor
$\log n$ in the third summation. Therefore, it is reasonable to
expect that the convergence rate of Clenshaw-Curtis quadrature would
be slightly slower than $\mathcal{O}(n^{-s-2})$.
\end{remark}

\section{Extrapolation methods for accelerating Clenshaw-Curtis quadrature}\label{sec:acceleration}
Comparing with Gauss quadrature, an essential feature of
Clenshaw-Curtis quadrature is that its quadrature nodes are nested.
This implies that the previous function values can be stored and
reused when the number of quadrature nodes is doubled. Therefore,
Clenshaw-Curtis quadrature is a particularly ideal candidate for
implementing an automatic quadrature rule in practical computations.
In this section, we shall extend our analysis in Section 2 and show
that it is possible to accelerate the rate of convergence of
Clenshaw-Curtis quadrature for functions with endpoint
singularities.

Asymptotic expansion of the error of Gauss-Legendre quadrature for
functions with endpoint singularities has been investigated in
\cite{sidi2009gauss,verlinden1997gauss}. For example, when $f(x) = (
1 - x)^{\alpha} g(x)$ with $\Re(\alpha) > -1$ and $g(x)$ is analytic
in a region containing the interval $[-1,1]$, Verlinden proved that
the error of the $n$-point Gauss-Legendre quadrature admits the
following asymptotic expansion \cite[Thm.~1]{verlinden1997gauss}
\begin{equation}
E_{n}^{G}[f] \sim \sum_{k = 1}^{\infty} c_k h^{k+\alpha}, \quad n
\rightarrow\infty,
\end{equation}
where $h = (n + 1/2)^{-2}$ and $c_k$ are constants independent of
$n$. Furthermore, some extrapolation schemes were proposed to
accelerate the convergence of Gauss quadrature. Even Verlinden's
results reveal an important connection between Gauss quadrature and
extrapolation schemes. However, accelerating the Gauss quadrature is
expensive since its quadrature nodes are completely distinct if $n$
is changed.

In the following, we shall show that the error of Clenshaw-Curtis
quadrature also admits a similar expansion. This together with the
nested property of Clenshaw-Curtis points implies that
Clenshaw-Curtis quadrature is more advantageous than its
Gauss-Legendre counterpart. Since the Clenshaw-Curtis points are
nested when $n$ is doubled, we restrict our attention to the case of
even $n$.

\begin{theorem}\label{thm:expansion error clenshaw}
If the Chebyshev coefficients of $f(x)$ satisfy
\begin{equation}\label{eq:asymp expan cheby coeff 1}
a_n \sim \sum_{k = 0}^{\infty} \frac{\mu_k}{ n^{d_k} }, \quad
n\rightarrow\infty,
\end{equation}
or
\begin{equation}\label{eq:asymp expan cheby coeff 2}
a_n \sim (-1)^n \sum_{k = 0}^{\infty} \frac{\mu_k}{ n^{d_k} }, \quad
n\rightarrow\infty,
\end{equation}
where $\mu_k$ are constants independent of $n$ and $0 < d_0 < d_1 <
\cdots$. Then, for even $n$, the error of the Clenshaw-Curtis
quadrature can be expanded as
\begin{equation}\label{eq:expansion error clenshaw}
E_{n}^{C}[f] \sim \sum_{k = 0}^{\infty}  \sum_{j=0}^{\infty}
\frac{\varsigma_{j,k} }{n^{d_k+2j+1}},
\end{equation}
where $\varsigma_{j,k}$ are constants independent of $n$.
\end{theorem}
\begin{proof} We only prove the case \eqref{eq:asymp expan cheby coeff
1} since the case \eqref{eq:asymp expan cheby coeff 2} can be proved
similarly. According to \eqref{eq:error of cc} and \eqref{eq:asymp
expan cheby coeff 1}, we have
\begin{align}
E_n^{C}[f] &= \sum_{m \in \Delta(n) } a_m E_n^{C}(T_m)
\nonumber \\
& \sim \sum_{m \in \Delta(n) } \left( \sum_{k = 0}^{\infty}
\frac{\mu_k}{ m^{d_k} }  \right) E_n^{C}(T_m) \nonumber \\
& = \sum_{k=0}^{\infty} \mu_k \sum_{m \in \Delta(n) }
\frac{1}{m^{d_k}} E_n^{C}(T_m).
\end{align}
Moreover, from \eqref{eq:aliasing} we have
\begin{align}
\sum_{m \in \Delta(n) } \frac{1}{m^{d_k}} E_n^{C}(T_m)  = \sum_{m
\in \Delta(n) } \frac{2}{m^{d_k} (1 - m^2) } + \sum_{m \in \Delta(n)
} \frac{2}{m^{d_k}(4r^2 - 1)}.
\end{align}
In the following we shall analyze the asymptotic of these two sums
on the right hand side of the above equation. For the first sum, it
is easy to see that
\begin{align}
\sum_{m \in \Delta(n) } \frac{2}{m^{d_k} (1 - m^2) } & = -2 \sum_{j
= 0}^{\infty} \sum_{m \in \Delta(n) } \frac{1}{m^{d_k+2j+2} }
\nonumber \\
& = \frac{2}{n^{d_k} (n^2 - 1)} - \sum_{j= 0}^{\infty}
\frac{1}{2^{d_k+2j+1}} \zeta\left(d_k+2j+2, \frac{n}{2}\right) ,
\end{align}
where $\zeta(s,a)$ is the Hurwitz zeta function. Recall the
asymptotic expansion of $\zeta(s,a)$
\cite[p.~25]{magnus1966formulas}
\begin{equation*}
\zeta(s,a) \sim  \frac{1}{(s-1)a^{s-1}} + \frac{1}{2a^{s}} +
\frac{1}{\Gamma(s)} \sum_{\ell=1}^{\infty}
\frac{B_{2\ell}}{(2\ell)!} \frac{\Gamma(s+2\ell-1)}{
a^{2\ell+s-1}},\quad a \rightarrow\infty,
\end{equation*}
it follows that
\begin{align}\label{eq:asymptotic error part one}
\sum_{m \in \Delta(n) } \frac{2}{m^{d_k} (1 - m^2) } &\sim
\frac{1}{n^{d_k} (n^2 - 1)} - \sum_{j=0}^{\infty}
\sum_{\ell=0}^{\infty} \frac{4^{\ell} B_{2\ell}
\Gamma(2\ell+2j+d_k+1) }{ (2\ell)! \Gamma(d_k+2j+2)
n^{2\ell+2j+d_k+1} }.
\end{align}
For the second sum, by means of the estimate of $S_2$ with $c(s) =
1$ and $s+1$ replaced by $d_k$, we see that
\begin{align}
\sum_{m \in \Delta(n) } \frac{2}{m^{d_k}(4r^2 - 1)}  &=
\frac{2}{(2n)^{d_k}} \sum_{j=1}^{\infty} \frac{1}{j^{d_k}} \sum_{
1-n \leq 2r \leq n} \frac{1}{ 4r^2 - 1} \left( 1 + \frac{r}{jn} \right)^{-d_k} \nonumber \\
& =  \frac{2}{(2n)^{d_k}} \bigg\{ - \left( \frac{2^{d_k}-1}{n^2-1} +
\frac{1}{n+1} \right)\zeta(d_k) \nonumber \\
&~~~ + \frac{1}{n^2-1} \sum_{j=1}^{\infty} \frac{1}{j^{d_k}}
\sum_{q=1}^{\infty} \frac{ (d_k)_{2q} }{ (2q)! (2j)^{2q} } \nonumber
\\
&~~~ + \left( \frac{n(n+2)}{n+1} - \frac{n^2}{n^2 - 1} \right)
\sum_{\ell = 1}^{\infty} \frac{(d_k)_{2\ell}
\zeta(2\ell+d_k)}{(2\ell)!
(2n)^{2\ell}} \nonumber \\
& ~~~ + \sum_{i=0}^{\infty} \frac{1}{n^{2i+1}} \sum_{j=1}^{\infty}
\frac{1}{j^{d_k}} \sum_{\ell=0}^{\infty} \frac{
\nu_{2i+1}^{i+\ell+1} (d_k)_{2\ell+2i+2} }{2^{2\ell} (2\ell+2i+2)!
j^{2\ell+2i+2} } \bigg\}. \nonumber
\end{align}
Observe that
\begin{align}
\sum_{j=1}^{\infty} \frac{1}{j^{d_k}} \sum_{q=1}^{\infty} \frac{
(d_k)_{2q} }{ (2q)! (2j)^{2q} } &= \sum_{j=1}^{\infty}
\frac{1}{j^{d_k}} \sum_{q=0}^{\infty} \frac{ (d_k)_{2q} }{ (2q)!
(2j)^{2q} } - \zeta(d_k) \nonumber \\
& = \frac{1}{2} \sum_{j=1}^{\infty} \frac{1}{j^{d_k}} \left( \left(1
+ \frac{1}{2j} \right)^{-d_k} +  \left( (1 - \frac{1}{2j}
\right)^{-d_k} \right) - \zeta(d_k) \nonumber \\
& = 2^{d_k}\zeta(d_k) - 2\zeta(d_k) - 2^{d_k - 1}.
\end{align}
Consequently,
\begin{align}\label{eq:asymptotic error part two}
\sum_{m \in \Delta(n) } \frac{2}{m^{d_k}(4r^2 - 1)}  &=
\frac{2}{(2n)^{d_k}} \bigg\{ - \left( \frac{1}{n+1} + \frac{1}{n^2 -
1} \right) \zeta(d_k) - \frac{2^{d_k - 1}}{n^2 - 1} \nonumber \\
&~~~  +  \left( \frac{n(n+2)}{n+1} - \frac{n^2}{n^2 - 1} \right)
\sum_{\ell = 1}^{\infty} \frac{(d_k)_{2\ell}
\zeta(2\ell+d_k)}{(2\ell)!
(2n)^{2\ell}} \nonumber \\
& ~~~ + \sum_{i=0}^{\infty} \frac{1}{n^{2i+1}} \sum_{j=1}^{\infty}
\frac{1}{j^{d_k}} \sum_{\ell=0}^{\infty} \frac{
\nu_{2i+1}^{i+\ell+1} (d_k)_{2\ell+2i+2} }{2^{2\ell} (2\ell+2i+2)!
j^{2\ell+2i+2} } \bigg\}.
\end{align}
Combining this with \eqref{eq:asymptotic error part one} gives
\begin{align}\label{eq:asymptotic error expansion}
\sum_{m \in \Delta(n) } \frac{1}{m^{d_k}} E_n^{C}(T_m) & \sim -
\sum_{j=0}^{\infty} \sum_{\ell=0}^{\infty} \frac{4^{\ell} B_{2\ell}
\Gamma(2\ell+2j+d_k+1) }{ (2\ell)! \Gamma(d_k+2j+2)
n^{2\ell+2j+d_k+1} } \nonumber \\
&~~~ - \frac{2^{1-d_k}}{n^{d_k}} \left( \frac{1}{n+1} + \frac{1}{n^2
- 1} \right) \zeta(d_k) \nonumber \\
&~~~ + \frac{2^{1-d_k}}{n^{d_k}} \left( \frac{n(n+2)}{n+1} -
\frac{n^2}{n^2 - 1} \right) \sum_{\ell = 1}^{\infty}
\frac{(d_k)_{2\ell} \zeta(2\ell+d_k)}{(2\ell)!
(2n)^{2\ell}} \nonumber \\
&~~~ + \frac{2^{1-d_k}}{n^{d_k}} \sum_{i=0}^{\infty}
\frac{1}{n^{2i+1}} \sum_{j=1}^{\infty} \frac{1}{j^{d_k}}
\sum_{\ell=0}^{\infty} \frac{ \nu_{2i+1}^{i+\ell+1}
(d_k)_{2\ell+2i+2} }{2^{2\ell} (2\ell+2i+2)! j^{2\ell+2i+2} }.
\end{align}
Since
\begin{equation*}
\frac{1}{n+1} + \frac{1}{n^2 - 1} = \sum_{j=0}^{\infty}
\frac{1}{n^{2j+1}}, \quad  \frac{n(n+2)}{n+1} - \frac{n^2}{n^2 - 1}
= n - \sum_{j=0}^{\infty} \frac{1}{n^{2j+1}}.
\end{equation*}
Thus, we can deduce that the asymptotic series on the right hand
side of \eqref{eq:asymptotic error expansion} consists of negative
powers of $n$ with exponents $\{ d_k + 2j + 1\}_{j=0}^{\infty}$ and
$k \geq 0$. This completes the proof.
\end{proof}

\begin{corollary}\label{cor: mixed asymptotic error}
If the Chebyshev coefficients of $f(x)$ satisfy
\begin{equation}\label{eq: mixed asymptotic}
a_n \sim  \sum_{k = 0}^{\infty} \frac{\mu_k}{ n^{d_k} } + (-1)^n
\sum_{k = 0}^{\infty} \frac{\gamma_k}{ n^{\zeta_k} }, \quad n
\rightarrow\infty,
\end{equation}
where $\mu_k, \gamma_k$ are constants independent of $n$ and $\{ d_k
\}_{k=0}^{\infty}$ and $\{ \zeta_k \}_{k=0}^{\infty}$ are positive
and strictly increasing sequences. Then, we have
\begin{equation}\label{eq:mixed error expansion clenshaw}
E_{n}^{C}[f] \sim \sum_{k = 0}^{\infty}  \sum_{j=0}^{\infty}
\frac{\varsigma_{j,k} }{n^{\xi_k+2j+1}}, \quad n \rightarrow\infty,
\end{equation}
where $\varsigma_{j,k}$ are constants independent of $n$ and $\{
\xi_k \}_{k = 0}^{\infty}$ is a strictly increasing sequence and $\{
\xi_k \}_{k=0}^{\infty} = \{ d_k \}_{k=0}^{\infty} \cup  \{ \zeta_k
\}_{k=0}^{\infty} $.
\end{corollary}
\begin{proof}
It follows from Theorem \ref{thm:expansion error clenshaw}.
\end{proof}

\begin{remark}\label{eq:d0 and s}
A direct consequence of Theorem \ref{thm:expansion error clenshaw}
is that the rate of convergence of Clenshaw-Curtis quadrature is
$\mathcal{O}(n^{-d_0-1})$. For example, if $f \in X^s$ which implies
that $d_0 = s + 1$. In this case, we can deduce immediately that the
rate of convergence of Clenshaw-Curtis quadrature is
$\mathcal{O}(n^{-s-2})$.
\end{remark}

For functions $f(x) = (1-x)^{\alpha}(1+x)^{\beta}g(x)$ with $\alpha,
\beta \geq 0$ are not integers simultaneously and $g(x) \in
C^{\infty}[-1,1]$, from Theorem \ref{thm:asymptotic albebraic
singularities} we know that their Chebyshev coefficients admit the
asymptotic of the form \eqref{eq:asymp expan cheby coeff 1} or
\eqref{eq:asymp expan cheby coeff 2} if $\beta$ or $\alpha$ is an
nonnegative integer. If both $\alpha$ and $\beta$ are not
nonnegative integers, then their Chebyshev coefficients admit the
asymptotic of the form \eqref{eq: mixed asymptotic}. Similarly, for
functions with algebraic-logarithmic singularities of the form
\eqref{eq:log singularity} with $\alpha$ a positive integer. If
$\beta$ is a nonnegative integer, then from \eqref{eq:logarithm one}
we see that the asymptotic of their Chebyshev coefficients satisfies
the form \eqref{eq:asymp expan cheby coeff 1}. If $\beta$ is not a
nonnegative integer, then the asymptotic of their Chebyshev
coefficients satisfies the form \eqref{eq: mixed asymptotic}.
Therefore, for these cases we mentioned, the error of the
Clenshaw-Curtis quadrature always has the asymptotic expansion of
the form \eqref{eq:expansion error clenshaw} or \eqref{eq:mixed
error expansion clenshaw}.

The error of the form \eqref{eq:expansion error clenshaw} or
\eqref{eq:mixed error expansion clenshaw} is especially suitable for
using some convergence acceleration techniques such as Richardson
extrapolation and $\epsilon$-algorithm to accelerate the convergence
rate of Clenshaw-Curtis quadrature. In particular, the previous
function evaluations can be reused in the process of convergence
acceleration when $n$ is doubled. In the following we only consider
the form \eqref{eq:expansion error clenshaw} since the form
\eqref{eq:mixed error expansion clenshaw} can be dealt with in a
similar way. In Algorithm $1$ we outline the main steps of the
convergence acceleration of Clenshaw-Curtis quadrature by using
Richardson extrapolation:
\begin{algorithm}\label{Richardson extrapolation}
\caption{Richardson extrapolation for Clenshaw-Curtis quadrature}
\begin{algorithmic}[1]
 \STATE {\mbox{Input parameters $n$ and $q$}}
 \FOR{$k=0:q$}
        \STATE {\mbox{Compute $R(0,2^kn) = I_{2^kn}^{C}[f]$ by FFT;}}
 \ENDFOR
 \FOR{$j = 0: q-1 $}
     \FOR{ $ k = 0:q-1-j $}
        \STATE { \mbox{ Evaluate
                ${\displaystyle R(j+1, 2^{k}n) = \frac{ 2^{d_j + 1} R(j, 2^{k+1}n) - R(j,2^{k}n) }{2^{d_j + 1} -
                1}; }$ } }
     \ENDFOR
 \ENDFOR
 \STATE{\mbox{Return $R(q,n)$. }}
\end{algorithmic}
\end{algorithm}

The term $R(q,n)$ achieves a higher order of convergence. More
precisely, from the standard theory of Richardson extrapolation we
have the following estimate
\begin{equation}
I[f] - R(q,n) = \mathcal{O}(n^{-d_q-1}) .
\end{equation}
Note that the Richardson extrapolation scheme $R(q,n)$ reduces to
Clenshaw-Curtis quadrature when $q = 0$.

\begin{corollary}\label{cor: order by exrapolation}
When using the Algorithm 1, the sequence $\{ d_k \}_{k=0}^{\infty}$
can be defined as follows: For functions $f(x) =
(1-x)^{\alpha}(1+x)^{\beta}g(x)$ with $\alpha, \beta \geq 0$ are not
integers simultaneously and $g(x) \in C^{\infty}[-1,1]$, we can
define
\begin{align}\label{eq:order algebraic singularities}
\{ d_k \}_{k=0}^{\infty} =  \left\{\begin{array}{ccc}
                                           {\displaystyle \{ 2\alpha+2j+1 \}_{j = 0}^{\infty} \cup \{ 2\beta+2j+1 \}_{j =
0}^{\infty} },   & \mbox{if $\alpha,\beta$ are not integers}, \\
[10pt]
                                           {\displaystyle  \{ 2\alpha+2j+1 \}_{j = 0}^{\infty} }, & \mbox{if $\beta$ is an integer}, \\ [10pt]
                                           {\displaystyle  \{ 2\beta+2j+1 \}_{j =
0}^{\infty}  },  & \mbox{if $\alpha$ is an integer}.
                                        \end{array}
                                        \right.
\end{align}
For functions $f(x) = (1 - x)^{\alpha} (1 + x)^{\beta} \log(1 - x)
g(x)$ where $\alpha$ is a positive integer, $\beta \geq 0$ and $g(x)
\in C^{\infty}[-1,1]$, then we can define
\begin{align}\label{eq:order logarithmic singularities}
\{ d_k \}_{k=0}^{\infty} = \left\{\begin{array}{ccc}
                                           {\displaystyle   \{ 2\alpha+2j+1 \}_{j = 0}^{\infty} }, & \mbox{if $\beta$ is an
integer}, \\ [10pt]
                                           {\displaystyle   \{ 2\alpha+2j+1 \}_{j = 0}^{\infty} \cup \{ 2\beta+2j+1 \}_{j =
0}^{\infty} },  & \mbox{otherwise}.
                                        \end{array}
                                        \right.
\end{align}
\end{corollary}

\begin{example}
Consider $f(x) = (1 - x)^{\alpha} g(x)$ and $\alpha>0$ is not an
integer. From \eqref{eq:class s} we know that $f \in X^s$ and $s =
2\alpha$. On the other hand, from Corollary \ref{cor: order by
exrapolation} we see immediately that $d_j = 2j + 2\alpha + 1$ for
$j \geq 0$. Thus, the convergence rate of the Richardson
extrapolation scheme $R(q,n)$ is
\begin{equation}
I[f] - R(q,n)  = \mathcal{O}(n^{-2q - s - 2}), \quad  q \geq 0.
\end{equation}
This higher order convergence rate is confirmed by numerical
experiments in the next section.
\end{example}

\begin{remark}\label{eq:interior singularities}
If $f(x)$ has an interior singularity inside the interval $[-1,1]$.
For example, suppose that
\begin{equation}
f(x) = (1 - x)^{\alpha} (1 + x)^{\beta} |x - x_0|^\delta g(x),
\end{equation}
where $x_0 \in (-1,1)$ and $\delta \geq 0$ is not an integer. Then,
we can first divide the interval $[-1,1]$ into two parts at $x =
x_0$ and then apply Clenshaw-Curtis quadrature or its extrapolation
acceleration scheme to the resulting two integrals.
\end{remark}

\section{Numerical experiments}\label{sec:example}
In this section we present some concrete examples to show the
convergence rates of Clenshaw-Curtis quadrature and Richardson
extrapolation approach for functions with endpoint singularities. We
apply ``Acceleration one" and ``Acceleration two" to indicate
$R(1,n)$ and $R(2,n)$, respectively. For comparison, we also add the
rate of convergence of Gauss-Legendre quadrature to the following
examples.

\begin{example}
Consider the following function
\begin{equation}\label{eq:test fun 1}
f(x) = (1 - x)^{\alpha} (1 + x)^{\beta} e^x,
\end{equation}
where $\alpha, \beta \geq 0$ are not integers simultaneously.
Obviously, Theorem \ref{thm:asymptotic albebraic singularities}
implies that $f(x) \in X^s$ where $s$ is defined as in
\eqref{eq:class s}. From Remark \ref{eq:d0 and s} and Corollary
\ref{cor: order by exrapolation} we know that the rate of
convergence of Clenshaw-Curtis quadrature is $\mathcal{O}(n^{-d_0 -
1})$ where $d_0 = s+1$, while the rate of convergence of the
Richardson extrapolation scheme $R(q,n)$ is $I[f] - R(q,n) =
\mathcal{O}(n^{-d_q-1})$ and $d_q$ is defined as in \eqref{eq:order
algebraic singularities}. Numerical results are illustrated in
Figure \ref{fig:branch} with two different choices of $\alpha$ and
$\beta$. The left graph of Figure \ref{fig:branch} demonstrates the
case $\alpha = \frac{1}{2}$ and $\beta = 0$ which implies $d_j =
2j+2$ for $j \geq 0$. The right graph of Figure \ref{fig:branch}
demonstrates the case $\alpha = \frac{3}{4}$ and $\beta =
\frac{1}{4}$. From \eqref{eq:order algebraic singularities} we can
deduce that $d_j = j + \frac{3}{2}$ for $j \geq 0$. It can be
observed clearly from Figure \ref{fig:branch} that the rate of
convergence of Clenshaw-Curtis quadrature is $\mathcal{O}(n^{-s-2})$
and the rate of convergence of the extrapolation scheme $R(q,n)$ is
$\mathcal{O}(n^{-d_q - 1})$ for $q = 1,2$, which coincides with our
analysis.

%

\end{example}

\begin{figure}[ht]
\centering
\includegraphics[width=6.5cm]{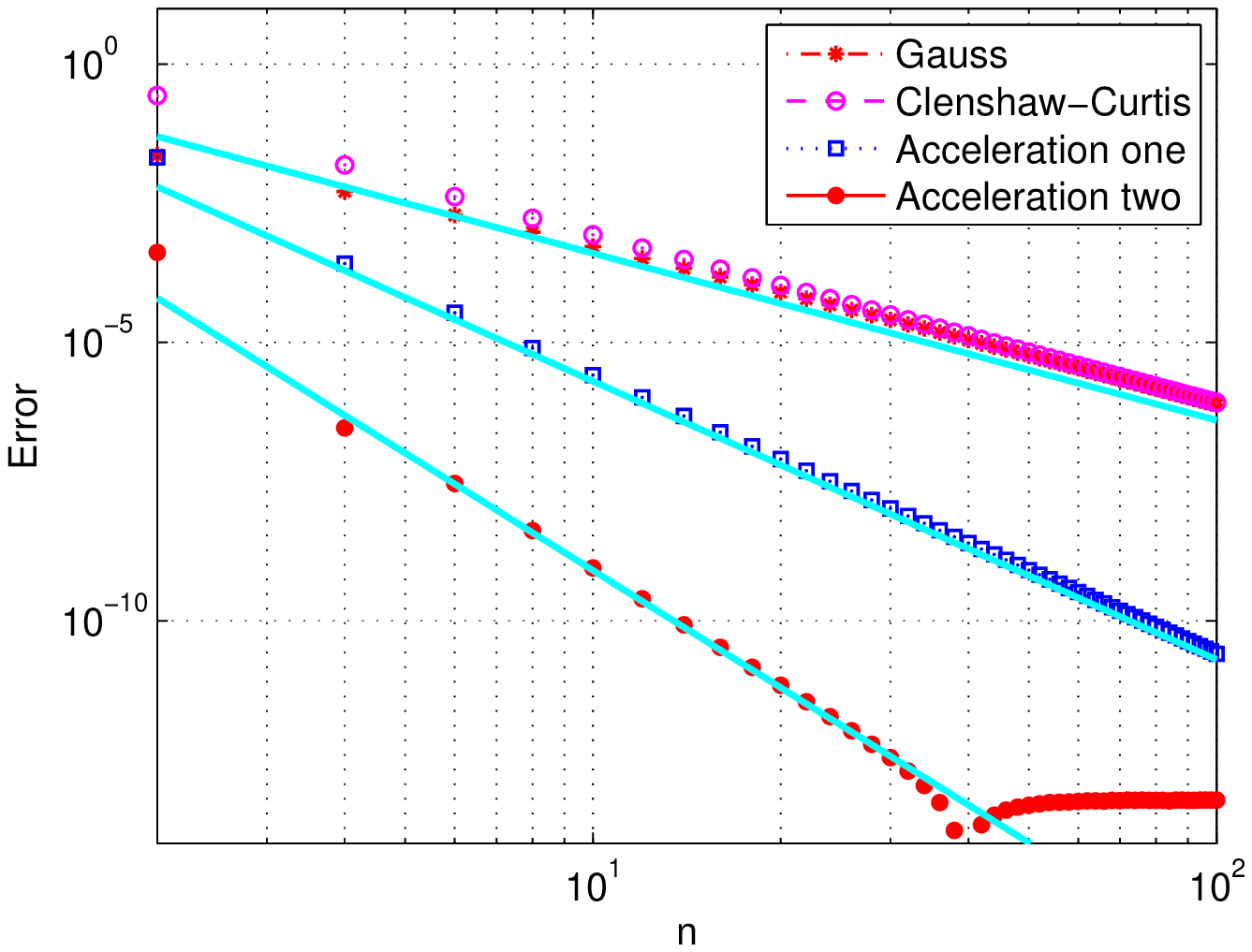}~~
\includegraphics[width=6.5cm]{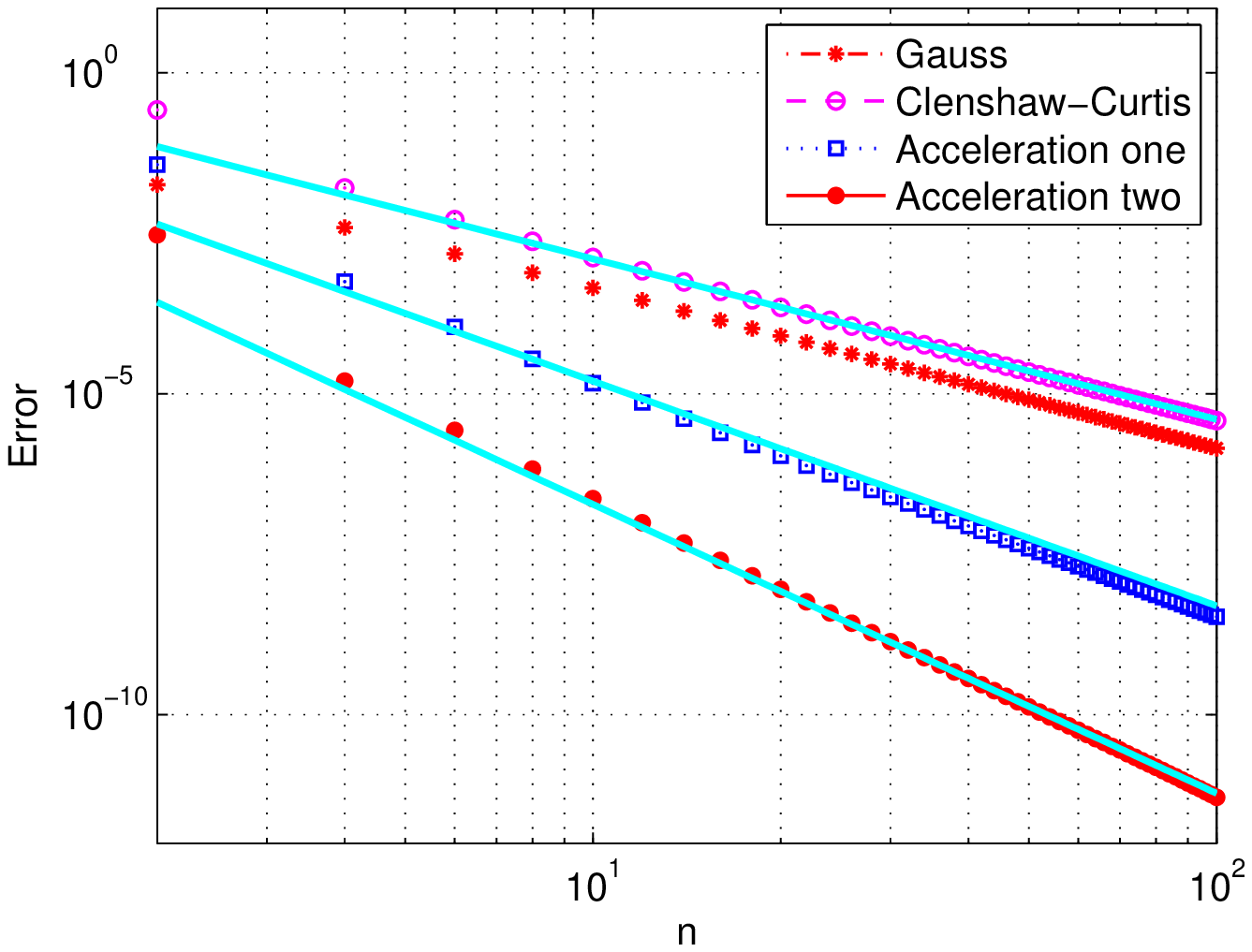}
\caption{Convergence rates of $(n+1)$-point Clenshaw-Curtis and
Gauss quadrature rules for $f(x) = (1-x)^{\alpha} (1+x)^{\beta} e^x
$ with $\alpha=\frac{1}{2}, \beta = 0$ (left) and
$\alpha=\frac{3}{4}, \beta = \frac{1}{4}$ (right). These lines
denote $\mathcal{O}(n^{-d_q-1})$ for $q=0$ (upper), $q=1$ (middle)
and $q=2$ (lower), and $d_q$ is defined as in \eqref{eq:order
algebraic singularities}.} \label{fig:branch}
\end{figure}

\begin{example}
Consider the function
\begin{equation}
f(x) = (1-x)^{\alpha} (1 + x)^{\beta}\log(1-x) \cos(t+1),
\end{equation}
where $\alpha$ is a positive integer and $\beta \geq 0$. Clearly, $f
\in X^s$ and $s$ is defined as in Remark \ref{def: s two}. In Figure
\ref{fig:logrithm} we demonstrate the convergence rate of
Clenshaw-Curtis and Gauss-Legendre quadrature rules and the
Richardson extrapolation schemes $R(1,n)$ and $R(2,n)$. The left
graph of Figure \ref{fig:logrithm} demonstrates the case $\alpha =
1$ and $\beta = 0$. In this case, we have from Definition \ref{def:
s two} and Corollary \ref{cor: order by exrapolation} that $s = 2$
and $d_j = 2j + 2\alpha + 1$ for $j \geq 0$. The right graph of
Figure \ref{fig:logrithm} demonstrates the case $\alpha = 1$ and
$\beta = \frac{1}{2}$. In this case, we have from \eqref{eq:order
logarithmic singularities} that $s = 1$ and $d_j = j+2$ for $j \geq
0$. The numerical results shown in Figure \ref{fig:logrithm}  are
consistent with our theoretical results.
\end{example}

\begin{figure}[ht]
\centering
\includegraphics[width=6.5cm]{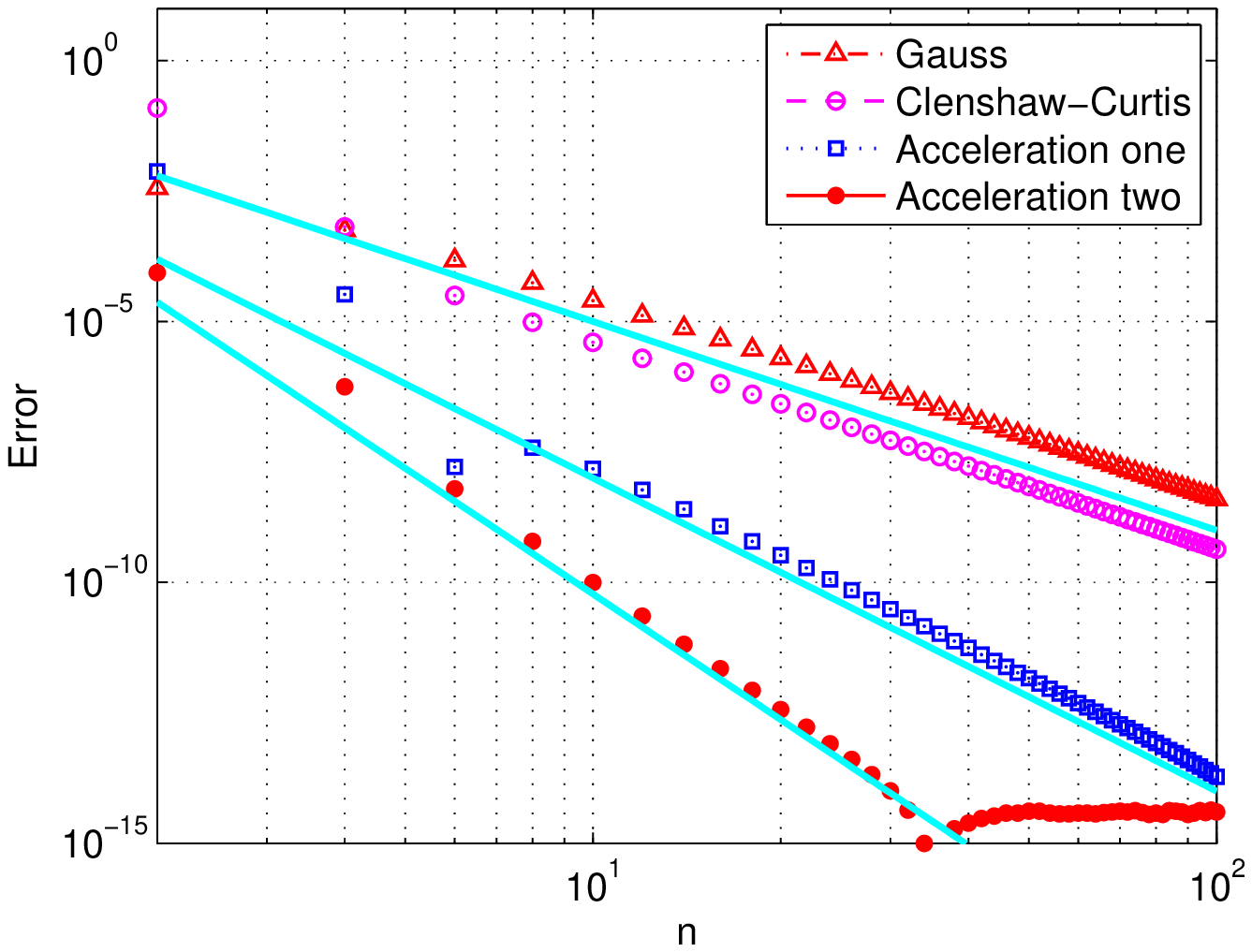}~~
\includegraphics[width=6.5cm]{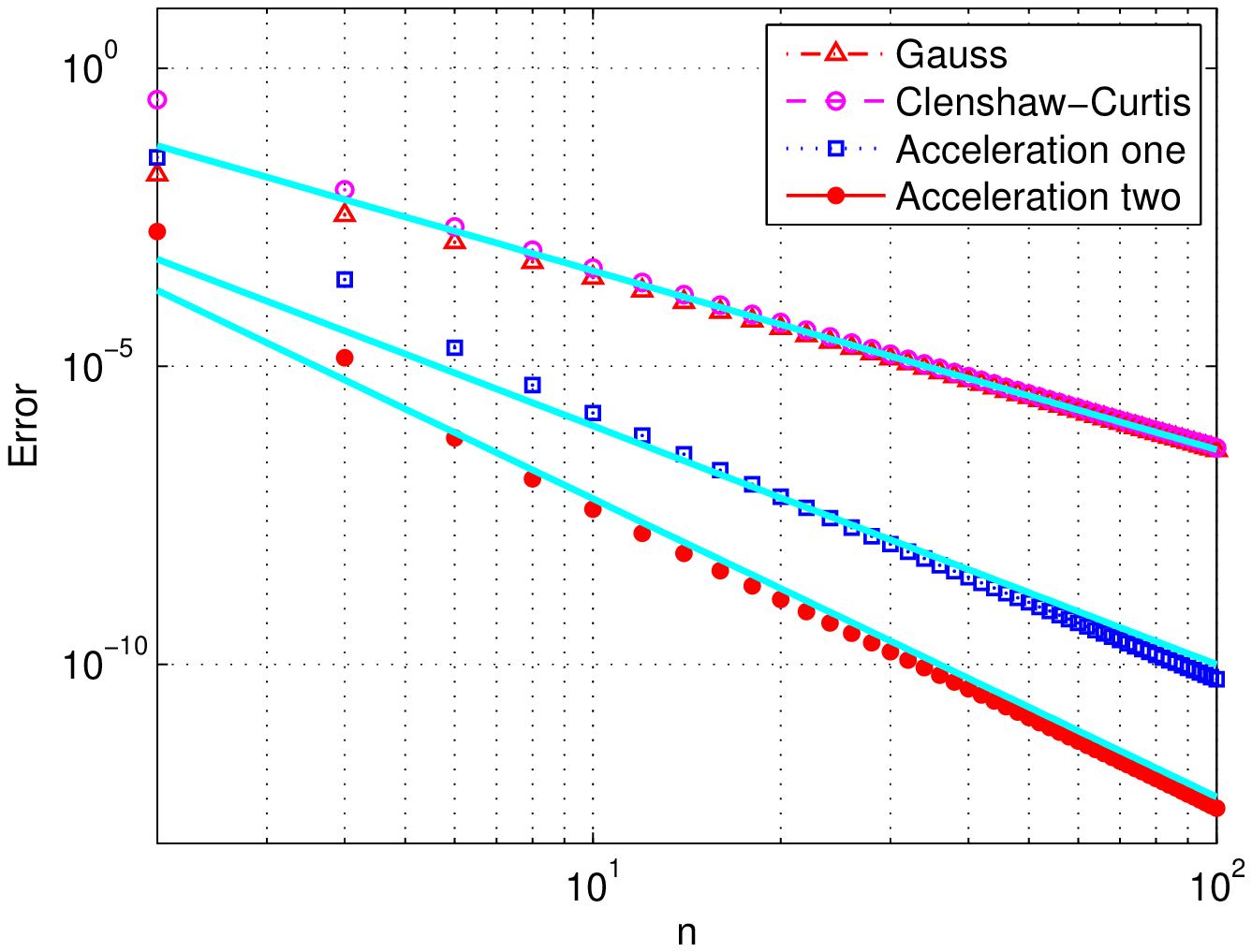}
\caption{Convergence rates of $(n+1)$-point Clenshaw-Curtis and
Gauss quadrature rules for $f(x) = (1-x)^{\alpha} (1 + x)^{\beta}
\log(1-x) \cos(t+1)$ with $\alpha = 1, \beta = 0$ (left) and $
\alpha = 1, \beta = \frac{1}{2}$ (right). The line denotes
$\mathcal{O}(n^{-d_q-1})$ with $q= 0$ (upper), $q=1$ (middle) and $q
= 2$ (lower).} \label{fig:logrithm}
\end{figure}

\begin{example}\label{example: arecos}
Finally, consider the function
\begin{equation}\label{eq:arccos}
f(x) = \arccos(x^{2m}),
\end{equation}
where $m$ is a positive integer. Using repeated integration by
parts, we obtain the asymptotic of its Chebyshev coefficients
\begin{equation*}
a_{2n} \sim \sum_{j=0}^{\infty} \frac{\mu_j}{n^{d_j}},
\end{equation*}
where $d_j = 2j+2$ for $j\geq 0$ and $\mu_j$ are constants depend on
$m$. Here we give explicit expressions for the first three
coefficients of $\mu_j$
\begin{equation}
\mu_0 = - \frac{\sqrt{2m}}{\pi}, ~~~~~ \mu_1 = -
\frac{\sqrt{2m}}{4\pi}\left(m - \frac{1}{2}\right), ~~~~~ \mu_2 =
-\frac{\sqrt{2m}}{16\pi}\left( m^2 - 5m + \frac{9}{4}\right).
\end{equation}
Moreover, $a_{2n+1} = 0$ for $n\geq0$ since the function $f(x)$ is
even. Obviously, $f \in X^1$ and it satisfies the condition of the
Theorem \ref{thm:superconvergence of cc}. Thus, the convergence rate
of the $(n+1)$-point Clenshaw-Curtis quadrature is
$\mathcal{O}(n^{-3})$. Figure \ref{fig:arccos} shows the convergence
rates of the Clenshaw-Curtis and Gauss-Legendre quadrature rules and
the Richardson extrapolation schemes $R(q,n)$ for the function
\eqref{eq:arccos} with two different values of $m$. Clearly, we can
see that the convergence rates of both quadrature rules are
$\mathcal{O}(n^{-3})$. Moreover, the convergence rate of $R(q,n)$ is
$\mathcal{O}(n^{-d_q-1})$.
\end{example}

\begin{figure}[ht]
\centering
\includegraphics[width=6.5cm]{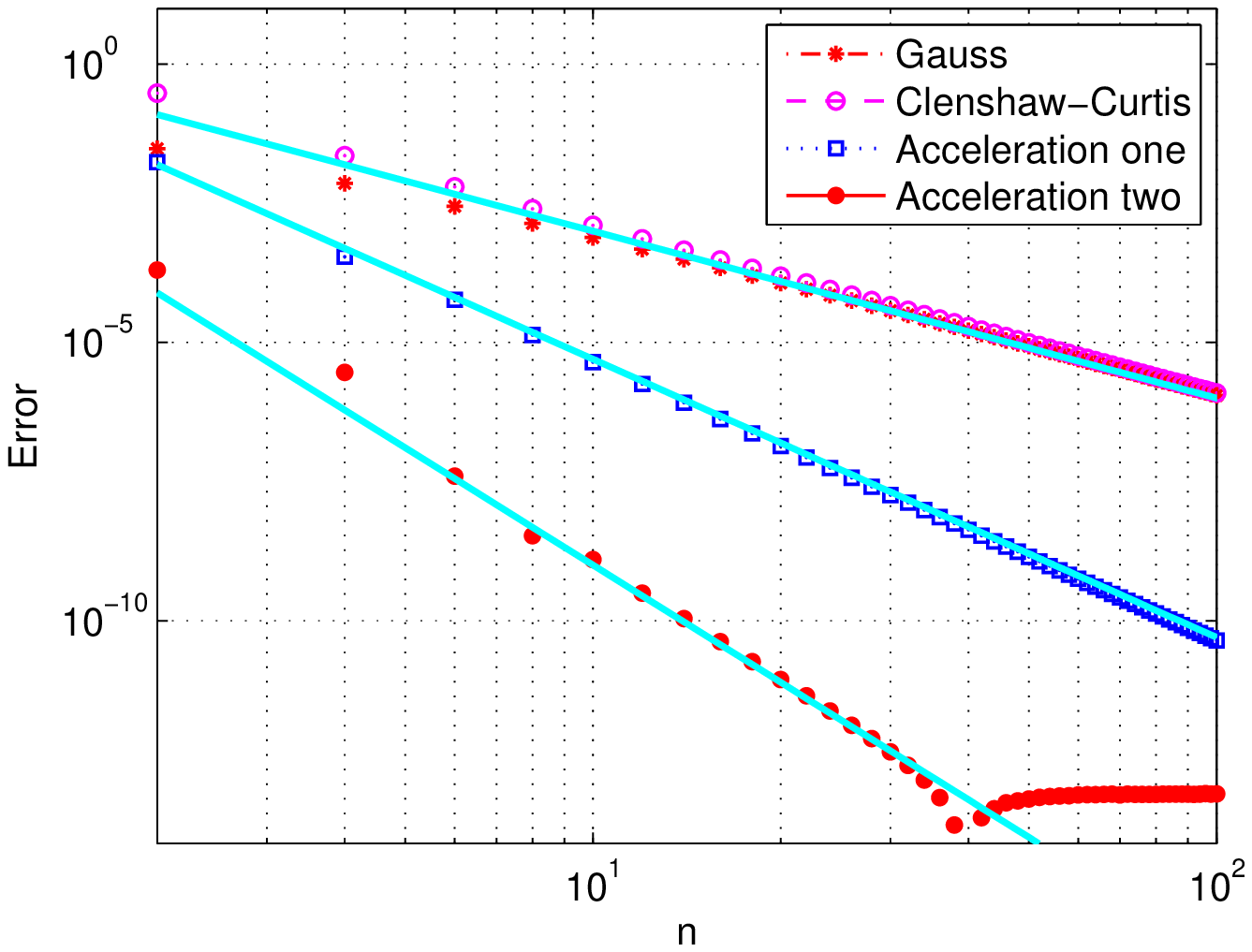}~~
\includegraphics[width=6.5cm]{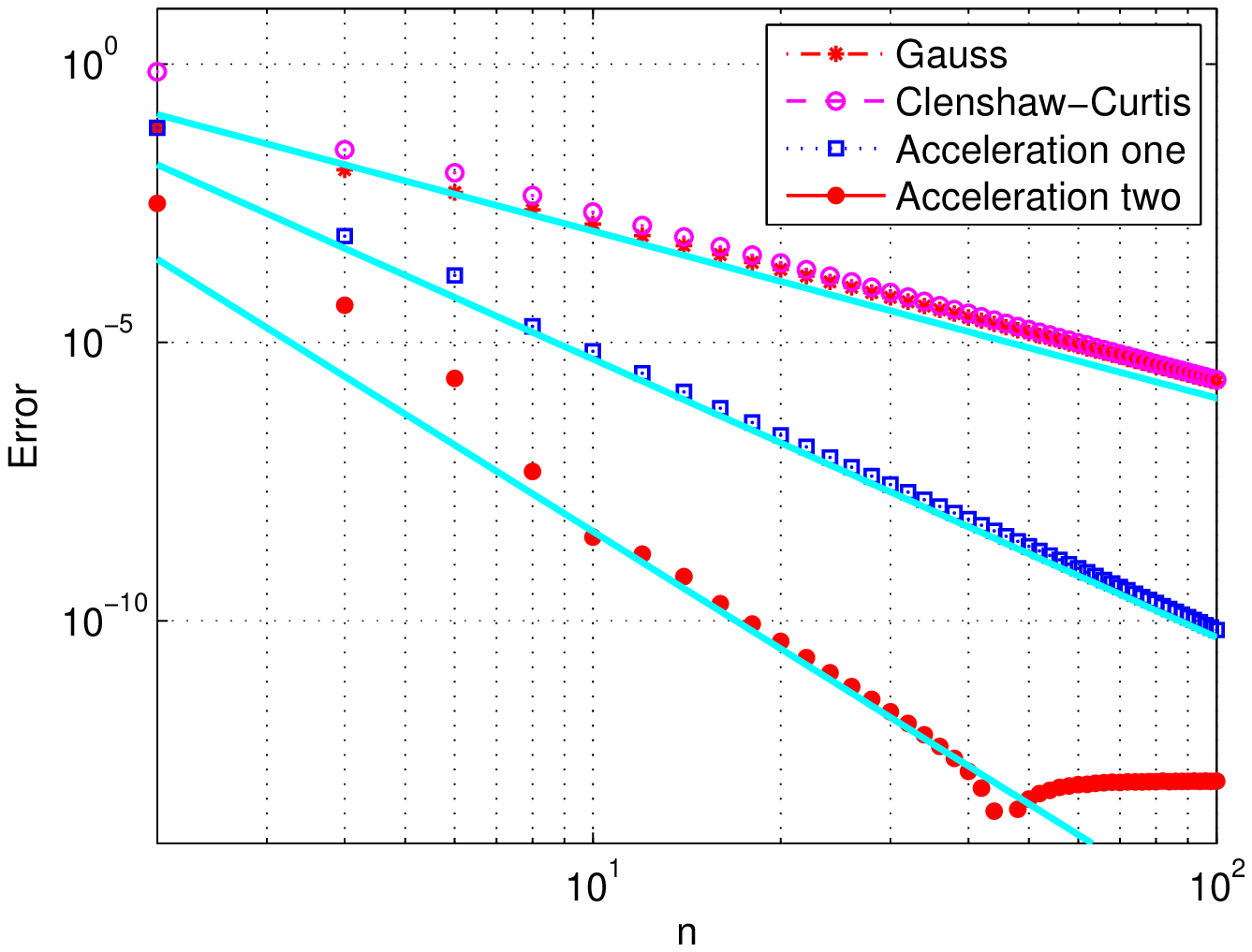}
\caption{Convergence rates of $(n+1)$-point Clenshaw-Curtis and
Gauss quadrature rules and the extrapolation schemes $R(q,n)$ for
$f(x) = \arccos(x^{2m})$ with $m=1$ (left) and $m=3$ (right). These
lines denote $\mathcal{O}(n^{-d_q-1})$ for $q=0$ (upper), $q=1$
(middle) and $q=2$ (lower).} \label{fig:arccos}
\end{figure}

\begin{remark}
From these examples, we can observe that the rate of convergence of
Gauss quadrature is almost indistinguishable with that of
Clenshaw-Curtis quadrature for functions with endpoint singularities
for large $n$.
\end{remark}

\section{Conclusion}\label{sec:conclusion}
In this paper, we have analyzed the rate of convergence of
Clenshaw-Curtis quadrature for functions in $X^s$ which have
algebraic or algebraic-logarithmic endpoint singularities. For such
functions, we show that the rate of convergence can be further
improved to $\mathcal{O}(n^{-s-2})$, which is one power of $n$
better than the optimal estimate given in
\cite{xiang2012clenshawcurtis}. Furthermore, an asymptotic error
expansion for Clenshaw-Curtis quadrature was obtained, based on
which extrapolation schemes such as Richardson extrapolation was
applied to accelerate the convergence of Clenshaw-Curtis quadrature.
In contrast to Gauss-Legendre quadrature, Clenshaw-Curtis quadrature
is a more powerful scheme to integrate functions with endpoint
singularities since its nodes are nested and its quadrature weights
can be evaluated efficiently by the inverse Fourier transform.

\bibliographystyle{amsplain}

\end{document}